\documentclass[12pt,fleqn,reqno]{amsart} 

\usepackage{amsmath,euscript}
\usepackage{amsthm}
\usepackage{amssymb}
\usepackage{graphicx}
\usepackage{mathrsfs}
\usepackage{epsf}
\usepackage{psfrag}
\usepackage{enumerate}
\usepackage{amsfonts,amssymb,amsthm,amsmath,latexsym,marginnote,color}

\usepackage{ulem}
\usepackage{cancel}
\usepackage{hyperref}

\usepackage[numbers]{natbib}
\usepackage{enumitem} 
\usepackage{xspace}
\usepackage[margin=1.3in]{geometry}

\usepackage{soul}

\newtheorem{thm}{Theorem}

\newtheorem{lemma}[thm]{Lemma}
\newtheorem{prop}[thm]{Proposition}
\newtheorem{cor}[thm]{Corollary}

\theoremstyle{remark}
\newtheorem{remark}[thm]{Remark}

\theoremstyle{definition}
\newtheorem{definition}[thm]{Definition}

\numberwithin{thm}{section}
\numberwithin{equation}{section}

\usepackage{soul}
\setstcolor{red}

\definecolor{green}{rgb}{0.0, 0.5, 0.5}

\definecolor{lgray}{gray}{0.9}
\definecolor{llgray}{gray}{0.95}
\definecolor{lllgray}{gray}{0.975}

\newcommand{\black}{\color{black}}

\newcommand{\magenta}{\color{magenta}}

\newcommand{\cjf}[1]{\magenta {\tt [JF: #1]} \black }
\newcommand{\jf}[1]{\magenta { #1} \black }

\newcommand{\nc}{\newcommand}

\nc{\la}{\label}
\nc{\ba}{\begin{array}}
\nc{\ea}{\end{array}}
\nc{\bs}{\begin{split}}
\nc{\es}{\end{split}}

\newcommand{\R}{\mathbb{R}}
\newcommand{\C}{\mathbb{C}}

\newcommand{\cF}{\mathcal{F}}

\nc{\al}{\alpha}
\nc{\del}{\delta}
\nc{\h}{\delta}
\nc{\G}{\Gamma}
\nc{\et}{\eta} 
\nc{\g}{\gamma}
\nc{\gam}{\gamma}
\nc{\ka}{\kappa}
\nc{\lam}{\lambda}
\nc{\Lam}{\Lambda}
\nc{\Om}{\Omega}
\nc{\om}{\omega}

\nc{\ta}{\tau}
\nc{\w}{\omega}
\nc{\io}{\iota}
\nc{\z}{\zeta}
\nc{\s}{\sigma}
\nc{\Si}{\Sigma}
\nc{\vphi}{\varphi}

\nc{\e}{\epsilon}

\newcommand{\veps}{\varepsilon}

\nc{\bP}{\bar{P}}
\nc{\bQ}{\bar{Q}}

\nc{\ran}{\rangle}
\nc{\lan}{\langle}

\newcommand{\ra}{\rightarrow}

\newcommand{\ls}{\lesssim}

\newcommand{\one}{\mathbf{1}}

\newcommand{\supp}{\operatorname{supp}}

\newcommand{\re}{\operatorname{Re}}
\newcommand{\im}{{\rm Im}}

\nc{\bfone}{{\bf 1}}
\nc{\1}{{\bf 1}}

\newcommand{\p}{\partial}

\newcommand{\n}{\nabla}

\def\jt{\langle t \rangle}

\newcommand{\DETAILS}[1]{}

\newcommand{\x}{\lan x\ran}
\newcommand{\bt}{\lan t\ran}

\nc{\den}{\text{den}}
\nc{\ex}{\text{xc}}
\nc{\Ex}{\text{Xc}}

\nc{\jx}{\langle x \rangle}

\begin{document}

\title[Maximal Velocity of Propagation for HE]{Maximal Speed of Quantum Propagation \\ for the Hartree equation}

\author{J. Arbunich}
\address{Jack Arbunich,  Department of Mathematics, University of Toronto, 40 St. George street, Toronto, 
  M5S 2E4, Ontario, Canada}
\email{jack.arbunich@utoronto.ca}

\author{J. Faupin}
\address{Jeremy Faupin, Institut Elie Cartan de Lorraine, Universit\'e de Lorraine, 57045 Metz Cedex 1, France} \email{jeremy.faupin@univ-lorraine.fr}
\email{}

\author{F. Pusateri}
\address{Fabio Pusateri,  Department of Mathematics, University of Toronto, 40 St. George street, Toronto, 
  M5S 2E4, Ontario, Canada}
\email{fabiop@math.toronto.edu}

\author{I. M. Sigal} 
\address{Israel Michael Sigal, Department of Mathematics, University of Toronto, 40 St. George street, Toronto,   M5S 2E4, Ontario, Canada}
\email{im.sigal@utoronto.ca}



\begin{abstract} 
We prove maximal speed estimates for nonlinear quantum propagation
in the context of the Hartree equation. 
More precisely, under some 
regularity and integrability assumptions 
on the pair (convolution) potential, we construct a set of energy and space 
localized initial conditions such that, up to time-decaying tails, solutions starting in this set 
stay within the light cone of the corresponding initial datum.
We quantify precisely the light cone speed, 
and hence the speed of nonlinear propagation, in terms of the momentum of the initial state.
\end{abstract}

\maketitle



\date{June 16, 2021}

\section{The Problem and Results}\label{sec:probl-res}

In contrast to the key principle of relativity, 
in Quantum Mechanics, a local change of initial conditions effects the solutions everywhere instantaneously. 
As Quantum Mechanics is the  theory of   the quantum matter in non-extreme  conditions,   
it is  important to understand the  limitations this imposes   
and to investigate  universal properties of the Quantum Mechanical (QM) propagation in some depth.  

QM predictions are  probabilistic 
 and therefore so is the  characterization of quantum evolution. It was shown in \cite{SigSof}  that, 
 assuming the energy is bounded initially, 
the supports of solutions of  the Schr\"odinger equation, up to vanishing in time probability tails, 
spread with a finite speed. This result was improved in \cite{HeSk,Skib,ArbPusSigSof} 
and extended in \cite{BonyFaupSig} to photon interacting with an atomic or molecular system, 
while \cite{FLS1, FLS2} developed a related approach in condensed matter physics.

For an electron in an atom the maximal velocity of propagation obtained in \cite{ArbPusSigSof} 
is close to the one computed heuristically and to the one observed: of the order of $10^5$m/s  
which is less than $0.1\%$ of the speed of light. Hence, in a few body case, 
the QM predictions here are fairly reliable.
 
Though bounds obtained in \cite{SigSof, HeSk,Skib} and especially in \cite{ArbPusSigSof} 
are valid for rather general class of potentials, for many-particles systems, 
their dependence on the number of particles (which in applications could vary
from a few to $10^{20}$) is rather poor.

 In this paper, we prove 
   bounds on  the speed of propagation 
for bosonic many-body systems in the mean-field approximation, i.e. for the Hartree equation (HE), 
\begin{align}\label{HE}
i\frac{\partial}{\partial t} \psi_{t} =(-\frac{1}{2}\Delta + V) \psi_{t} + v* |\psi_t|^2 \psi_{t}, \end{align}
 in the physical space $\R^d, d \ge 3$. 
Here  $V$ and $v$ are real functions, external (`one-body') and internal (interparticle, or `pair') potentials.  

We make the following two sets of assumptions on the pair convolution potential $v$,
and the external potential $V$:

\smallskip

\setlength{\leftmargini}{1.5em}
\begin{itemize}
\item 
We assume that $v:\R^d\to\R$ satisfies
\footnote{The regularity assumptions on $v$ 
are probably not optimal and can be weakened.
We however decided to state things in this way for technical convenience.}
\begin{align}\label{v-cond}
\begin{split}
& v \in W^{\gamma,(q,\infty)}  \text{ if } 1 < q < d/2,  
\quad \text{for}  \quad d/(2q)<\gamma,
\\
& \mbox{or} \quad v \in W^{\gamma,1} 
  \quad \text{with}  \quad \gamma > d/2,
\end{split}
\end{align}
where $W^{\gamma,q}$ is the standard Sobolev space ($W^{\gamma,2}\equiv H^\gamma$), 
and $W^{\gamma,(q,\infty)}$ denotes the Lorentz-Sobolev space; see \eqref{LorentzSobSp} below.

\smallskip
\item
We assume that $V:\R^d\to\R$ satisfies Yajima-type conditions\footnote{Note
that these conditions are not optimal, and we opted to state them in 
a simple and concise form that applies to all dimensions $d\geq 3$.
We refer to the cited papers for more precise assumptions.}
(cf. \cite{Yajima,FincoYajima,RodSchl}, \cite[Eq. (2.10)]{GHW}, and references therein):
\begin{align}\label{Aas1}
\begin{split}
&| \x^{\sigma} \langle \nabla\rangle^{\alpha} V | \lesssim 1, \qquad \sigma \geq \tfrac{3}{2}d + 3, \quad 
  \alpha \leq \gamma + \lfloor \tfrac{d-1}{2} \rfloor,
\end{split}
\end{align}
where $\langle x \rangle := \sqrt{1+|x|^2}$.

\smallskip
\item Furthermore, we let $H:=-\frac{1}{2}\Delta + V$ and assume that
\begin{align}
\label{Aas2'}
&  \text{$H$ has neither nonpositive eigenvalues nor a resonance at } 0.
\end{align}
\end{itemize}

We recall that $H:=-\frac{1}{2}\Delta + V$ is said to have a resonance at $0$ if the equation $Hu=0$ 
has a distributional solution $u\in \langle x\rangle^{\nu}L^2
= \{f=\langle\cdot\rangle^\nu g \, | \, g\in L^2\}$, for any $\nu>\frac12$.  
One can show that there are no zero energy resonances in dimensions $d \ge 5$.
In Proposition \ref{prop:condV} below we give explicit restrictions
on $V$ for which condition \eqref{Aas2'} is satisfied.

The assumptions \eqref{Aas1}-\eqref{Aas2'} are more than enough to guarantee
that the propagator $e^{-itH}$ is bounded on the Sobolev space $H^\gamma$,
and satisfies the same dispersive estimates as the propagator $e^{it \Delta/2}$; 
see \eqref{propagHsbound} and \eqref{eq:disp}.

\smallskip
Denote $\chi = \chi(r)$ for a smooth bump function supported in $[0,1]$ and such that $\chi=1$ on $[0,1/2]$. 
Let $\chi_A$ denote the characteristic function of a set $A$. 
For $I\subset\R$ a bounded open interval, we define the upper speed 
(or momentum, as the mass of particles is set equal to $1$) bound given that the energy 
of the initial state $\psi_0$ is supported in $I$ as
\begin{align}\label{kg}
k_I:= {\||\n|\chi_{I}(H)\|}.
\end{align}
Finally, let $\mathcal{B}^s(\varepsilon)$ denote the ball in $H^s$ centered at the origin and of radius $\varepsilon>0$.
With this, we formulate our main result.

\begin{thm}[Maximal propagation speed for HE]\label{thm:max-vel-HE0} 
Assume Conditions \eqref{v-cond}, \eqref{Aas1} and \eqref{Aas2'}.
Let $I$ be a bounded open interval, $g \in C^\infty_0(I;\R)$,  
$b>0$  and $s\ge\gamma$.  
Then, there exist $\varepsilon > 0$ and a set 
\begin{equation}\label{subset_S} 
S_{g, b} \subset \mathcal{B}^s(C\varepsilon), 
\end{equation}
  for some absolute $C>0$ (see \eqref{setSdef}),
  such that the following hold true: 

\smallskip
(i) $S_{g, b}$ is in one-to-one correspondence with
 the set  $g(H) \chi(|x|/b) \mathcal{B}^s(\varepsilon)$;

\smallskip
(ii) For any initial condition $\psi_0 \in S_{g, b}$,  the Hartree equation \eqref{HE} has a global solution in $H^{s}$
and this solution satisfies the estimate 
\begin{align}\label{max-vel-est-nl0}
\| \chi_{\{ |x| \geq c |t|+a \}}  \psi_t\| \ls \varepsilon \langle t \rangle^{-1/2},
\end{align}
for any constants $c$ and $a$ satisfying $c>k_I$ and $a>b$. 
\end{thm}

\begin{remark}[On the decay rate]
The algebraic decay rate of $-1/2$ in \eqref{max-vel-est-nl0} can probably be improved.
The restriction is coming from the presence of the nonlinearity, as one can see by comparing with
the arbitrary decay rate obtained in the linear (time-independent) case in \cite{ArbPusSigSof}.
\end{remark}

\begin{remark}[Scaling]
In the quantum mechanical context, one would like to keep $\|\psi_0\|^2$, interpreted as the initial number of particles, at least of $O(1)$. To reflect this interpretation, one may phrase the main result of Theorem \ref{thm:max-vel-HE0} without assuming smallness on the data,
but assuming instead that the pair interaction potential is small.
In other words, we may let $\varepsilon = 1$ in the statement of Theorem \ref{thm:max-vel-HE0}
and change $v$ to $\varepsilon^2 v$ in the Hartree equation \eqref{HE}.
\end{remark}

\smallskip
By definition  \eqref{setSdef} below, 
 the set $S_{g, b}$  consists of $H^s$-functions `localized' in the position $x$ 
 and energy $H$  (essentially to $|x|\le b$ and $I$, respectively). 
  Estimate \eqref{max-vel-est-nl0} shows that (up to time-decaying tails), 
  a solution to the Hartree equation starting in the set
\[S:=\bigcup_{g\in C^\infty_0, b>0}S_{g, b}\] 
 stays inside the $c$-light cone of an initial condition. After \cite{SigSof}, we call such a result a
{\it maximal (propagation)  speed bound} (MSB).

The infimum of all $c$'s for which \eqref{max-vel-est-nl0} holds is called 
the {\it maximal speed of propagation}, $c_{\rm max}$. 
Theorem \ref{thm:max-vel-HE0} implies the bound $c_{\rm max}\le k_I$.

As was mentioned above, for the (linear) Schr\"odinger equation ($v=0$), 
such an estimate was found in \cite{SigSof}, improved in \cite{HeSk,Skib,ArbPusSigSof}.

The MSB is conceptually close to the celebrated  Lieb-Robinson bound in  Quantum Statistical Mechanics proved 
  in \cite{LR} and improved 
 by many authors and extended to various areas of quantum physics with many important applications  
  (see \cite{exp1, GEi, NachSim} for reviews and references).\footnote{The  
  Lieb-Robinson bound does not involve an explicit energy cut-off, 
  but there is an implicit one in setting up the problem on a lattice, rather than a continuous 
     space.}

For certain models of quantum many-body systems
  without the Hartree or Hartree-Fock approximation, the Lieb-Robinson and maximal velocity bounds were obtained in a number of papers, see \cite{ElMaNayakYao, GebNachReSims, MatKoNaka, FLS1, FLS2, Fossetal, NRSS, exp1, SHOE, WH, KS, YL} and references therein.

 Our approach follows that of  \cite{ArbPusSigSof} (originating in turn in 
  \cite{SigSof})  for the linear Schr\"o\-dinger equation 
 and is based on constructing nearly monotonic (for bounded energy intervals) 
  `propagation observables' 
  providing quantitative information about the quantum evolution in question.  
 One of our key contributions 
  is the extension of the quantum energy localization method, 
 which proved to be exceptionally effective in the linear context, to nonlinear problems. 
 
\begin{remark}[MSB-intuition]\label{rem: MSB-intuit} 
The MSB is determined by the kinetic energy term and the main effort is directed at controlling the influence of the effective time-dependent potential, due to the nonlinearity (time-dependent self-consistent potential). Such a control is subtle as, a priori, one cannot rule out that the time-dependent potential does not pump energy onto the systems leading to an acceleration.
\end{remark}

\begin{remark}[Finiteness of $k_I$] \label{rem:V-Del-bnd}
Since $V$ is bounded (in fact, $\Delta$-boundedness with a relative bound $<1$ (see \eqref{V-cond} below) suffices),
$\Delta$ is $H$-bounded and therefore $k_I$ in \eqref{kg} is finite.
\end{remark}

\begin{remark}[Sharpness of the MSB]\label{rem: kappa-sharp}
Under our assumptions, the bound $c_{\mathrm{max}} \leq k_I$ is sharp. To illustrate this, we consider the special case of $v=0$. In this case, we can take  $S_{g, b}=g(H) \chi(|x|/b) L^2(\R^d)$. Then, one can show (see e.g. \cite{Sig}) that solutions of \eqref{HE} with $v=0$ (the Schr\"odinger equation) with initial conditions in $S_{g, b}$ concentrate on the trajectories $x=p t+x_0$, where $p=-i \n$ is the quantum momentum operator and $x_0\in \R^d$, and consequently $g(H)\approx g(\frac12 |p|^2 + V(p t+x_0))\approx g(\frac12 |p|^2)$, as $t\ra \infty$.  Choosing $I$ to be a small interval around an energy $E >0$ and using that  
 $g(H)$ is supported in $I$, we see that the support of the solution $\psi_t$ expands at the rate $\sqrt E t$. On the other hand $k_I= {\||p|\chi_{I}(H)\|}\approx \sqrt E$. (Of course, one can simplify the arguments by taking $V=0$.)  
 \end{remark}

\begin{remark}[Dependence on number of particles]\label{rem: n-dep} 
The bound $k_I$ is of the same form as in the (linear) 
one-particle case and gives a similar magnitude for the speed as in that case. 
However, the constant (subsumed in the symbol $\ls$) on the r.h.s. of 
\eqref{max-vel-est-nl0} depends on the size of the pair potential $v$. (In the mean-field derivation of the HE, one considers  the pair potential $w$ in the many-body SE to be of the form $w=\frac1n v$, with $v=O(1)$ in the particle number $n$ (see \cite{BenPorSchl}). This $v$ enters the HE.)
\end{remark}

\smallskip
\paragraph{\it Open problems.} 

(a) Prove MSB \eqref{max-vel-est-nl0} with  the implicit constant 
on the r.h.s.  independent of the size of the pair potential $v$. 
  
(b) Prove the MSB \eqref{max-vel-est-nl0} for $n$-particle systems with $n$-independent bounds.

(c)  Prove the MSB \eqref{max-vel-est-nl0} for  the Hartree-von Neumann (`mixed state' Hartree) equation, see \cite{PDdft} for some results on the long-time dynamics of this equation. 

(d)  Investigate a possibility of an acceleration due to strong nonlinearities.

\medskip
The following proposition gives explicit restrictions on $V$ guaranteeing that 
Condition \eqref{Aas2'} holds.

\begin{prop}\label{prop:condV}
Assume that the positive and negative parts of the potential $V$, $V_+=\max(V,0)$, $V_-=\max(-V,0)$, satisfy
\begin{equation}\label{Aas2-1}
V_+(x)\le C \x^{-\al}, \quad V_-(x)\le\delta\x^{-\al}, \quad \al>2,
\end{equation}
for some $C>0$ and for $\delta>0$ small enough. Then the operator $H=-\frac12\Delta+V$ has no eigenvalues and $0$ is not a resonance of $H$.
\end{prop}

Proposition \ref{prop:condV}, or a similar statement, might be known. 
For the convenience of the reader, we provide its proof 
in Appendix \ref{app:condV}.

Theorem \ref{thm:max-vel-HE0} and Proposition \ref{prop:condV} imply
\begin{cor}\label{cor:condV}
Assume that Conditions \eqref{v-cond}, \eqref{Aas1} and \eqref{Aas2-1} hold. 
 Then the conclusions of Theorem \ref{thm:max-vel-HE0} hold.
\end{cor}

\begin{remark}[The case of pure power NLS]\label{rem:gNLS}
Our main result, Theorem \ref{thm:max-vel-HE0},   
can be extended also to the case of the {\it pure power NLS equation}
\begin{align}\label{NLS0}
i\frac{\partial}{\partial t} \psi_{t} = (-\frac{1}{2}\Delta + V) \psi_{t} + |\psi_{t}|^{2\sigma} \psi_{t},
 \qquad \|\psi_0\|_{H^s \cap L^{p'}(\R^d)} \leq \varepsilon,
\end{align}
see Remark \ref{rem:gNLS-pf} below for a  detailed discussion.
\end{remark}

\medskip
\noindent
{\it Organization of the paper}. 
In Section \ref{sec:global-exist}, we state  a standard
 global existence result   with time decay for small solutions. In Section \ref{sec:ICset-S}, we describe the class of admissible data, $S=\bigcup_{g\in C^\infty_0, b>0}S_{g, b}$, for our MPS bounds. The main result of this section is proven in  Section \ref{sec:propic}, after we state in Section \ref{sec:mapping} some preliminary estimates. 
    In Section \ref{sec:prf-MVB-thm}, we prove
Theorem \ref{thm:max-vel-HE0}. Technical proofs are collected in the appendices.

\medskip
\noindent
{\it Notation}.
As mentioned above, $L^q$ and $W^{s,q}$ denote the usual Lebesgue and Sobolev spaces, $L^{q,\infty}$ 
denotes the weak-$L^q$ space, and $W^{s,(q,\infty)}$ ($q>1$) stands for Lorentz-Sobolev space, 
\begin{equation}\label{LorentzSobSp}
W^{\gamma,(q,\infty)}=\big\{ f\in L^{q,\infty} \, , \, \langle \nabla\rangle^{\alpha} f \in L^{q,\infty}, 
  \, \alpha \le \gamma \big\}.
\end{equation}
To simplify the statements below,
we will also write, abusing notations in the case where $q=1$,
\begin{equation*}
L^{1,\infty}\equiv L^1, \qquad W^{s,(1,\infty)} \equiv W^{s,1}.
\end{equation*}

We let $\langle x \rangle := \sqrt{1+|x|^2}$ for $x \in \R^d$,
and denote by $L^2_\gamma := \langle x \rangle^{-\gamma} L^2$ the usual weighted $L^2$ space.

We use standard notation for norms, 
and will often let $\| \cdot \| = \| \cdot \|_{L^2}$. 
We will sometimes also use $\| \cdot \|$ to denote the operator norm when there is no risk of confusion.

As usual, $p'$ denotes the H\"older conjugate exponent of $p \in [1,\infty]$. 
The Lebesgue indices $q$ and $p$ appearing in the statements of the results 
(see for example \eqref{v-cond-glob} and \eqref{ic} in Theorem \ref{thm:HE-exist}) are related as follows:
\begin{equation}\label{pqp1}
\frac{1}{p}=\frac12-\frac{1}{2q}, \qquad \frac{1}{p}  + \frac{1}{p'}= 1. 
\end{equation}

We write $A \lesssim B$ if $A \leq C B$ for some absolute constant $C>0$ independent of $A$ and $B$.
We write 
$A\approx B$ if $A\lesssim B$ and $B\lesssim A$.

More notation will be introduced in the course of the proofs.

\medskip

\begin{small}
\noindent \textbf{Acknowledgements.} We are grateful to the anonymous referees for their many constructive remarks and suggestions.
In particular, Remarks 1.3, 1.4, 1.6 and a part of 1.7 and Open Problems (c) and (d) were inserted in response to questions and comments by the referees. 

The research of JA and IMS was supported in part by NSERC grant 7901. FP was supported in part by a start-up grant from the University of Toronto, and NSERC grant
RGPIN-2018-06487.
\end{small}

\smallskip
\section{Global solutions}\label{sec:global-exist}

We begin with a basic theorem on the existence of global small solutions with time decay. 

\begin{thm}\label{thm:HE-exist}
Let $1\le q  < d / 2$ and  $s > d/(2q)$.  
Assume that $V$ satisfies  Conditions \eqref{Aas1}, \eqref{Aas2'} and that 
\begin{align}\label{v-cond-glob}
v \in L^{q,\infty} \quad \text{if}\quad 1 < q < d/2 ,
  \qquad v \in L^{1} \quad\text{if}\quad  q =1.
\end{align}
Then, there exists $\bar{\varepsilon}>0$ such that, for any $\varepsilon \leq \bar{\varepsilon}$,
Eq. \eqref{HE} with an initial condition $\psi_0$ satisfying   
\begin{align}\label{ic} 
\psi_0 \in H^s \cap L^{p'}, \qquad \| \psi_0 \|_{H^s\cap L^{p'}}  
\le \varepsilon,
\end{align}
 where $p'$ is given by \eqref{pqp1}, i.e. $p'= \frac{2q}{q +1}$ for $q\ge 1$, 
  has a unique global solution 
$\psi_t\in C(\R, H^s(\R^d))$ and this solution satisfies, for all $t\in \R$,
\begin{align}
\label{L2-ests}
& \|\psi_t\|_{L^2} =  \|\psi_0\|_{L^2},  \\
\label{Hs-ests}
& \|\psi_t\|_{H^s} \lesssim  \|\psi_0\|_{H^s}, 
\\ \label{pq-ests} 
& \|\psi_t\|_{L^{p}} \lesssim \varepsilon \langle t \rangle^{-d/(2q)}.
\end{align}
\end{thm}

Theorem \ref{thm:HE-exist} is proven using 
  standard energy and decay estimates. For $q=1$ ($v\in L^1$), it is proven in \cite{Dietze}. In Appendix \ref{sec:HE-exist}, 
we present a proof  for the full range of $q$'s.

\medskip
\section{Set of initial conditions and asymptotic energy cut-off}\label{sec:ICset-S}

In this section, we describe the set of initial conditions, 
$S=\bigcup_{g\in C^\infty_0, b>0}S_{g, b}$,
used in  Theorem \ref{thm:max-vel-HE0}. To do this, we introduce \textit{asymptotic energy cut-offs}.

Let  $f(|\psi|^2):=v*|\psi|^2$ and let $H_{t}^\psi$ be the $\psi_{t}$-dependent Schr\"odinger operator on $L^2(\R^d)$ defined as
\begin{align}\label{defH_t}
H_{t}^\psi := -\frac12 \Delta + V + f(|\psi_{t}|^2).
\end{align}
If $\psi_t \in L^p$, with $p$ as in Theorem \ref{thm:HE-exist}, 
by Young's inequality we have $f(|\psi_{t}|^2) \in L^\infty$.
Then, by Kato's result (see e.g. \cite{CFKS}), 
the operator $H_{t}^\psi$ is self-adjoint on the domain of $\Delta$.
We let $U_t^\psi\equiv U^\psi_{t, 0}$ be the propagator generated by $H_{t}^\psi$. 
Then a solution $\psi_t$ of \eqref{HE} satisfies the fixed point equation $\psi_t= U_t^\psi\psi_0$. 
We also denote
\begin{align}\label{defW_t}
W_t^\psi := f(|\psi_t|^2) = v* |\psi_t|^2, \qquad H_t^\psi=H+W_t^\psi.
\end{align}

For real-valued $g\in C_0^\infty(\R)$ 
we define the asymptotic energy cut-offs (cf. \cite{ArbPusSigSof, SigSof}) as the operator norm limit
\begin{align} \label{gtau}
g_+^\psi(H) :=  \lim_{\tau\ra \infty} (U_\tau^\psi)^{-1}g(H)U_\tau^\psi.
\end{align}
The next proposition will imply that this limit exists under our assumptions.

\begin{prop}
\label{prop:as-en-cutoff}
Let $W_t(x)=W(x, t)$ be a real, time-dependent bounded potential satisfying 
\begin{align}\label{Wt-cond0}
\int_0^\infty \|\p_x^\al W_t\|_{L^\infty} dt <\infty, \quad \text{where either}\quad
  \alpha=0\quad\text{or}\quad 1\le|\al|\le 2.
\end{align}
Let $U_t := U(t, 0)$ be the evolution generated by $H_t := H + W_t$. 
Then, for all $g\in C_0^\infty(\R;\R)$, the following operator-norm limit exists
\begin{align} \label{g+}
g_+(H) := \lim_{t\rightarrow \infty}  U_t^{-1}g(H)U_t.
\end{align}
\end{prop}
 
\begin{proof}
Define $g_t(H):=U_t^{-1}g(H)U_t$, 
and write $g_t(H)$ as the integral of the derivative and use that 
$\p_r g_r(H) = -i  U_r^{-1}[g(H), W_r]U_r$ to obtain
\begin{align} \label{gt}
g_t(H)=g(H) - i \int_0^t U_r^{-1} [g(H), W_r]U_rdr.
\end{align}
If \eqref{Wt-cond0} holds with $\alpha=0$, using the trivial estimate 
\begin{equation*}
\|[g(H),W_r]\| \lesssim \|g(H)\| \|W_r\|_{L^\infty} \lesssim \|W_r\|_{L^\infty},
\end{equation*} 
shows that \eqref{g+} exists.

If \eqref{Wt-cond0} holds with $1\le|\al|\le 2$, then we use Lemma \ref{lem:g-W-comm}, which shows that
\begin{equation*}
\big \| [g(H), W_r] \big \| \lesssim \max_{1\le|\alpha|\le2} \|\p_x^\al W_r \|_{L^\infty}.
\end{equation*}
Hence \eqref{g+} exists.
\end{proof}
\begin{remark}[Bound \eqref{Wt-cond0}]\label{rem: MVB-intuit} For either the Hartree case \eqref{defW_t} or the NLS one \eqref{NLS0Wt}, 
the bound \eqref{Wt-cond0} follows from the decay of $\|\psi_t \|_{L^p}$ at a sufficiently fast rate in $t$;  see for example Lemma \ref{lem:Wt}.
\end{remark} 
\begin{cor}[Existence of the asymptotic cutoff \eqref{gtau}]\label{rem:as-en-cutoff}
Under the conditions of Theorem \ref{thm:HE-exist}, if $\psi_t$ is a solution of \eqref{HE} 
then, for all $g\in C_0^\infty(\R;\R)$, the limit \eqref{g+} exists and
\begin{align}\label{g+expl}
g_+^\psi(H) = g(H) -i \int_0^\infty  \big(U_r^\psi\big)^{-1} [g(H), W_r^\psi] \, U^\psi_r \, dr .
\end{align}
\end{cor}

\begin{proof}
Under the conditions of Theorem \ref{thm:HE-exist}, Young's inequality implies that
\begin{align}\label{Wt-est}
&\| W_t^\psi\|_{L^\infty} \ls \bt^{-d/q},
\end{align}
see Lemma \ref{lem:Wt} below. Hence since $q<d$, Proposition \ref{prop:as-en-cutoff} implies the existence of the limit \eqref{gtau} and
the formula \eqref{g+expl}.
\end{proof}

To proceed, recall that $\mathcal{B}^s(\varepsilon)$ denotes the ball of radius $\varepsilon$ in $H^s$,
\begin{equation}\label{Bs}
\mathcal{B}^s(\varepsilon) := \big \{ f \in H^s, \, 
\|f\|_{H^s}<\varepsilon \big \}.
\end{equation}
Recall also that $L^2_\gamma:=\langle x \rangle^{-\gamma} L^2$ 
and let $\| f \|_{L^2_\gamma}:=\|\langle x\rangle^\gamma f\|_{L^2}$ be the corresponding norm. 
Note that if $\gamma > d/(2q)$, then $L^2_\gamma \subset L^{p'}$ 
where $p'$ is given by \eqref{pqp1}. For $\varepsilon>0$, 
 $\mathcal{B}_\gamma^s(\varepsilon)$ denotes the ball of radius $\varepsilon$ in $L^2_\gamma\cap H^s$:
\begin{equation}\label{Bsg}
\mathcal{B}_\gamma^s(\varepsilon) := \big \{ f \in L^2_\gamma\cap H^s, \, 
\|f\|_{L^2_\gamma}+\|f\|_{H^s}<\varepsilon \big \}.
\end{equation}
The next proposition is the main result of this section.

\begin{prop}\label{prop:ic-ext21} 
Let $1<q<d/2$, 
 $d/(2q)<\gamma<d/q-1$, and $\gamma\le s$.
Let $g \in C_0^\infty(\R;\R)$. Let $\varepsilon,\tilde{\varepsilon}>0$ 
be such that $C_0\varepsilon = \tilde{\varepsilon}\ll1$ with $C_0$ sufficiently large. 
Then assuming Conditions \eqref{v-cond}, \eqref{Aas1} and \eqref{Aas2'} on the potentials $V$ and $v$,

\setlength{\leftmargini}{2em}
\begin{itemize}

\item For all $\phi\in\mathcal{B}_\gamma^s(\varepsilon)$
there exists a unique $\psi_0\in\mathcal{B}_\gamma^s(\tilde{\varepsilon})$ solving the equation  
\begin{align}\label{ic-fp}\psi_0=g_+^{\Psi (\psi_0)} (H) \phi,\end{align}
where $g_+^\psi (H)$ is defined in \eqref{gtau} and $\Psi (\psi_0): t\ra $ the solution 
$\psi_t$ of \eqref{HE} with 
the datum $\psi_0$ (which exists globally since $\psi_0\in\mathcal{B}_\gamma^s(\tilde{\varepsilon}) 
\subset \tilde{\varepsilon}(H^s\cap L^{p'})$ and satisfies 
the properties in Theorem \ref{thm:HE-exist}).

\smallskip
\item The map $\Phi_g: \phi \mapsto \psi_0$ restricted to the domain 
$g(H)\mathcal{B}^s_\gamma(\varepsilon)$ is injective 
(note that $g(H)\mathcal{B}^s_\gamma(\varepsilon)\subset\mathcal{B}^s_\gamma(C\varepsilon)$ for some $C>0$).

\smallskip 
\end{itemize}

\end{prop}

Eq. \eqref{ic-fp} is a fixed point problem depending on $\phi$ (or an implicit function equation). 
We prove Proposition \ref{prop:ic-ext21} in Section \ref{sec:propic}.

\begin{definition}[Set $S_{g, b}$]\label{def:setS}
Let $\Phi_g$ be the injective maps defined in the second statement of Proposition 
\ref{prop:ic-ext21}. For $g\in C^\infty_0(\R;\R)$ and $b>0$, we define the set $S_{g, b}$ 
of initial conditions, $\psi_0$, appearing in Theorem \ref{thm:max-vel-HE0}, Eq. \eqref{subset_S}, as
\begin{align}\label{setSdef}
S_{g, b}: = 
 \Phi_g\big(g(H)\chi(|x|/b)\mathcal{B}^s(\varepsilon)\big).\end{align}
\end{definition}

Proposition \ref{prop:ic-ext21} shows that the class of initial data $S_{g, b}$  
is in one-to-one correspondence with the set 
  $g(H) \chi(|x|/b) \mathcal{B}^s(\varepsilon)$ 
(note that $\chi(|x|/b) \mathcal{B}^s(\varepsilon)\subset \mathcal{B}_\gamma^s(C\varepsilon)$ for some $C>0$).
 Proposition \ref{prop:ic-ext21} and  Theorem \ref{thm:HE-exist} imply that
 
\begin{cor}\label{cor:setS} 
Under the conditions of Proposition \ref{prop:ic-ext21}, 
the set $S_{g, b}$ of initial conditions, $\psi_0$, constructed in \eqref{setSdef},
has properties (i) and the first part of (ii) (that for any initial condition $\psi_0 \in S_{g, b}$,  
the Hartree equation \eqref{HE} has a global solution in $H^{s}$) of Theorem \ref{thm:max-vel-HE0}.  
\end{cor}

\medskip
\section{Proof of Theorem \ref{thm:max-vel-HE0}}\label{sec:prf-MVB-thm}

We will use  the following result about the linear propagators proven in  \cite{ArbPusSigSof} 
 under stronger assumptions.

\begin{thm}[Maximal propagation speed for $t$-dependent potentials]\label{thm:max-vel-x-Ht}
Suppose that $H_t = -\frac12\Delta +V+ W_t$, with  $V(x)$ satisfying the inequality 
\begin{align}\label{V-cond} &  
 \|Vu\|  \le \frac{a_1}{2} \|\Delta u\|+ a_2\|u\|,  
 \end{align} 
for some $0 \le a_1 <1, \, a_2>0$,  
and $W_t(x)$, a real, time-dependent, bounded potential such that 
\begin{align}\label{eq:integr_wt}  
\begin{split}
&\text{either} \quad\int_0^\infty w_t \, dt < \infty \quad 
 \text{or}  \quad \int_0^\infty w_t' \, dt< \infty,
 \\
& \text{where}  \quad w_t :=\int_t^{\infty} \|W_r\|_{W^{1,\infty}} dr  \quad  \text{and}  \quad w_t' := 
\max_{1\le|\alpha|\le2} \int_t^{\infty} \|\p_x^\al W_r\|_{L^\infty} dr.
\end{split}
\end{align}

Let $I$ be a bounded open interval, $g\in C_0^\infty(I;\R)$ and let $k_I$ be as in \eqref{kg}.
If $c> k_I$ and $a>b$, then, for all $0<\beta<1$,  the evolution  $U_t = U(t, 0)$ 
generated by $H_t$ satisfies the estimate \begin{align} \label{max-vel-est-Ht}
\|\chi_{\{|x| \geq ct+a \}} \, U_t g_+(H) \, \chi_{\{|x| \leq b\}} \| \ls \,t^{-\min(\frac12,1-\beta)} + w_{t^\beta}^\sharp 
\end{align}
for $t\ge1$.
Here $w_{t^\beta}^\sharp$ is either $w_{t^\beta}$ or $w_{t^\beta}'$, depending on the condition in \eqref{eq:integr_wt}.
\end{thm} 
A proof of Theorem \ref{thm:max-vel-x-Ht} is given in Appendix \ref{app:max-vel_Ht}. 
Note that in our context, it is preferable to use the condition 
$\int_0^\infty\int_t^\infty\|W_r^\psi\|_{W^{1,\infty}}drdt<\infty$ since it requires 
less regularity on $v$ than the condition 
$\int_0^\infty\int_t^\infty\|\partial_\alpha W_r^\psi\|_{L^{\infty}}drdt<\infty$, $1\le|\alpha|\le2$.

\begin{proof}[Proof of Theorem \ref{thm:max-vel-HE0}]
Let $\psi_t\in C(\R, H^s(\R^d))$ be the unique global solution of Eq. \eqref{HE} with  
an initial condition $\psi_0$ in the set $S_{g, b}$ (see \eqref{setSdef}), given in Theorem \ref{thm:HE-exist}. 
Theorem \ref{thm:max-vel-HE0}  follows from Theorem \ref{thm:max-vel-x-Ht} by letting (see \eqref{defW_t})
\begin{align}
W_t \mapsto W_t^\psi := f(|\psi_t|^2), \qquad U_t \mapsto U_t^\psi, 
  \qquad \beta = 1/2,
\end{align}
provided we verify Condition \eqref{eq:integr_wt}, that $w_{t^\beta} \ls \langle t\rangle^{-1/2}$,
and Proposition \ref{prop:ic-ext21} which shows that the class of initial data $S_{g, b}$  
is in one-to-one correspondence with the set   $g(H) \chi(|x|/b) \mathcal{B}^s(\varepsilon)\subset g(H) \mathcal{B}_\gamma^s(C\varepsilon)$ for some $C>0$ (see Corollary \ref{cor:setS}). 

For Condition \eqref{eq:integr_wt}, we need 
the following direct consequence of Theorem \ref{thm:HE-exist}:
 
\begin{lemma}\label{lem:Wt}
Under the conditions of Theorem \ref{thm:HE-exist}, we have that
\begin{equation}\label{eq:Wt-decay}
\big\|W_t^\psi\big\|_{L^\infty} \lesssim\varepsilon^2\langle t \rangle^{-\frac{d}{q}}.
\end{equation}
Suppose in addition that $v$ satisfies \eqref{v-cond}. Then 
\begin{equation}\label{eq:Wt-decay2}
 \big\|W_t^\psi\big\|_{W^{\gamma,\infty}}\lesssim\varepsilon^2\langle t\rangle^{-\frac{d}{q}}.
\end{equation}
\end{lemma}

\begin{proof}[Proof of Lemma \ref{lem:Wt}]
First consider the $L^\infty$-norm of $W_t^\psi$. By Young's inequality,
\begin{equation*}
\big\|W_t^\psi\big\|_{L^\infty}\lesssim\|v\|_{L^{q,\infty}}\big\||\psi_t|^2\big\|_{L^{q'}}
\lesssim\|v\|_{L^{q,\infty}}\|\psi_t\|_{L^p}^2,
\end{equation*}
since $p=2q'$. Hence the first inequality in \eqref{eq:Wt-decay} follows from \eqref{pq-ests} in Theorem \ref{thm:HE-exist}. If we suppose in addition that $v$ belongs to $W^{\gamma,(q,\infty)}$, then we can write
\begin{equation*}
\big\|W_t^\psi\big\|_{W^{\gamma,\infty}}\lesssim\|v\|_{W^{\gamma,(q,\infty)}}\big\||\psi_t|^2\big\|_{L^{q'}}\lesssim\|v\|_{W^{\gamma,(q,\infty)}}\|\psi_t\|^2_{L^p},
\end{equation*}
and hence \eqref{eq:Wt-decay2} follows again from \eqref{pq-ests}.
\end{proof}

Now, condition \eqref{eq:integr_wt} follows from Lemma \ref{lem:Wt} since, using $\gamma\ge1$, 
we have $\|W_t^\psi\|_{W^{1,\infty}}\\ \lesssim \langle t\rangle^{-d/q}$ with $q<d/2$.
Hence, $w_t\ls \langle t\rangle^{-d/q+1}$ and Theorem \ref{thm:max-vel-HE0} follows. 
\end{proof}

\begin{remark}[Integrability assumptions]
In dimension $3$, the endpoint case in Theorem \ref{thm:HE-exist} for our integrability conditions 
is $v \in L^{3/2}$ and $\psi_0 \in L^{6/5}$. 
This would correspond to sharp decay in $L^6_x$
that implies that $\| W_t \|_{L^\infty} \lesssim t^{-2}$,
which is the borderline rate for the current argument;
see \eqref{eq:integr_wt}. 
\end{remark}

\begin{remark}[The case of pure power NLS]\label{rem:gNLS-pf}
To extend Theorem \ref{thm:max-vel-HE0} to the NLS \eqref{NLS0} (see Remark \ref{rem:gNLS}), we take  
\begin{align}\label{NLS0Wt}
W_t^\psi := |\psi_t|^{2\s}
\end{align}
instead of $v \ast |\psi_t|^2$.
It is not hard to show that there exist numbers $\sigma_0 = \sigma_0(d)$
and $s_0 = s_0(\s,d)$ sufficiently large, such that
assuming that $\s \geq \s_0(d)$ and $s\geq s_0(\s,d)$, 
a global existence result analogous to 
Theorem \ref{thm:HE-exist} holds for \eqref{NLS0},
and implies bounds on $W_t$
that can be used to verify condition \eqref{eq:integr_wt} and apply Theorem \ref{thm:max-vel-x-Ht}.
The analogue of Proposition \ref{prop:ic-ext21} needed to construct a suitable class of 
initial data, can also be proven following the same arguments we give in Section \ref{sec:propic},
and the analogous results from Section \ref{sec:mapping} for the flow of \eqref{NLS0}.

For example, in dimension $d=3$, consider $\sigma = 3/2$ (quartic NLS) and let $p'\in[1,6/5)$.
Then one can show (as in the proof of Theorem \ref{thm:HE-exist}) that
the same estimate \eqref{Hs-ests}-\eqref{pq-ests} hold for solution of \eqref{NLS0}:
\begin{align}\label{NLSp}
{\| \psi_t \|}_{L^p} \lesssim \varepsilon \jt^{-3(1/2-1/p)},
\qquad {\| \psi_t \|}_{H^s} \lesssim \varepsilon.
\end{align}
Note that $3(1/2-1/p) > 1$ (since $p \in (6,\infty]$).

Then, let us first fix, for the sake of explanation, $p'=1$.
It follows that
\begin{align}\label{NLS1}
\begin{split}
{\| W_t^\psi \|}_{W^{1,\infty}} = {\big\| |\psi|^{2\sigma} \big\|}_{W^{1,\infty}}
	\ls {\big\| \psi_t \big\|}_{W^{1,\infty}} {\| \psi_{t} \|}^2_{L^\infty}
	\ls \varepsilon^3 \jt^{-3},
\end{split}
\end{align}
provided $s>5/2$,
and, for all $\gamma \leq s$,
\begin{align}\label{NLS2}
\begin{split}
{\| W_t^\psi \|}_{H^\gamma} = {\big\| |\psi|^{2\sigma} \big\|}_{H^\gamma}
  \ls {\big\| \psi_t \big\|}_{H^\gamma} {\| \psi_{t} \|}^2_{L^\infty} 
  \ls \varepsilon^3 \jt^{-3}.
  \end{split}
\end{align}

In particular, the condition in Proposition \ref{prop:as-en-cutoff}, 
and the stronger condition \eqref{eq:integr_wt} in Theorem \ref{thm:max-vel-x-Ht} hold.
Analogues of Lemma \ref{lem:Wt} and \ref{lem:Wt2} 
also hold (note that the only relevant thing is that the exponent $d/q > 2$,
so we can fix $q = 3/2-\e$ when comparing to the rates in \eqref{NLS1}-\eqref{NLS2}).
The mapping properties in Lemma \ref{lemHs} and \ref{lemL2gamma} can also be proved using \eqref{NLSp},
and the same goes for the estimates on the differences from Lemmas \ref{lemLpdiff}-\ref{lemL2diff},
since these only rely on \eqref{NLSp}-\eqref{NLS2} and the above mentioned lemmas.
Theorem \ref{thm:max-vel-HE0} then follows for solutions of \eqref{NLS0} with $\s=3/2$ and $p' = 1$.

\end{remark}

\medskip
\section{Mapping properties of $W_t^\psi$ and $U_t^\psi$}\label{sec:mapping}
In this section, we state several  properties of $W_t^\psi$ and the flow $U^\psi_t$ 
that will be essential ingredients in the proofs in the next section. 
Proofs are deferred to Appendix \ref{ssec:use}.

\subsection{Mapping properties of $W_t^\psi$}
Recall that the norms of $W_t^\psi$ in the spaces $L^\infty$ and $W^{\gamma,\infty}$ 
have been estimated in Lemma \ref{lem:Wt_ab}. We also need to estimate the $W^{s,2q}$-norm of $W_t^\psi$.

\begin{lemma}\label{lem:Wt_ab}
Under the conditions of Theorem \ref{thm:HE-exist}, we have that
\begin{equation}\label{eq:Wt-decay_ab}
 \big\|W_t^\psi\big\|_{W^{s,2q}}\lesssim\varepsilon^2\langle t\rangle^{-\frac{d}{2q}}.
\end{equation}
\end{lemma}

Identifying $W_t^\psi$ with a multiplication operator, Lemmas \ref{lem:Wt} and \ref{lem:Wt_ab} imply the following
\begin{lemma}\label{lem:Wt2}
Under the conditions of Theorem \ref{thm:HE-exist}, we have that
\begin{equation}\label{eq:Wt_Hs}
\big\|W_t^\psi\big\|_{H^s\mapsto H^s} \lesssim \varepsilon^2 \langle t \rangle^{-\frac{d}{2q}}.
\end{equation}
Suppose in addition that $v$ satisfies \eqref{v-cond}. Then 
\begin{equation}\label{eq:Wt_Hs2}
\big\|W_t^\psi\big\|_{H^\gamma\mapsto H^\gamma} \lesssim \varepsilon^2 \langle t \rangle^{-\frac{d}{q}}.
\end{equation}
\end{lemma}

Lemmas \ref{lem:Wt} and \ref{lem:Wt2} show that by imposing stronger regularity conditions on $v$, 
namely $v\in W^{s,(q,\infty)}$ (see Condition \eqref{v-cond}) instead of $v\in L^{q,\infty}$, 
one improves the decay rate of $W_t^\psi$. 
This can also be achieved by assuming more regularity on the initial data $\psi_0$. 

\begin{lemma}\label{lem:Wt3}
Let $1\le q  < d / 2$ and $s>\gamma$. Let $\sigma=\sigma_1+\sigma_2$ with $\sigma_1,\sigma_2\ge0$ and
\begin{equation*}
s>\frac{d}{2q}+\frac{\sigma_2d}{d-2q}.
\end{equation*}
Assume that $V$ satisfies  Conditions \eqref{Aas1}, \eqref{Aas2'} and that
\begin{align}
v \in W^{\sigma_1,(q,\infty)} \quad \text{if}\quad 1 < q < d/2 ,
  \qquad v \in L^{\sigma_1,1} \quad\text{if}\quad  q =1.
\end{align}
Then, there exists $\bar{\varepsilon}>0$ such that, for any $\varepsilon \leq \bar{\varepsilon}$,
for any
\begin{align} 
\psi_0 \in H^s \cap L^{p'}, \qquad \| \psi_0 \|_{H^s\cap L^{p'}}\le \varepsilon, 
\end{align}
we have
\begin{equation*}
\big\| W_t^\psi\big\|_{W^{\sigma,\infty}}\lesssim\varepsilon^{2-2\varepsilon'q/d}\langle t \rangle^{-d/q+\varepsilon'},
\end{equation*}
with $d/q-\varepsilon'>2$. In particular,
\begin{equation*}
\big\|W_t^\psi\big\|_{H^s\mapsto H^s}\lesssim\varepsilon^{2-2\varepsilon'q/d}\langle t \rangle^{-d/q+\varepsilon'}.
\end{equation*}
\end{lemma}

In the next lemmas of this section, one can replace the regularity assumption on $v$ 
(i.e. $v$ satisfies Condition \ref{v-cond}) 
by the assumptions of Lemma \ref{lem:Wt3}. 
For simplicity, and since in our application we need that $v$ satisfies \eqref{v-cond}, we do not elaborate.

\subsection{Mapping properties of $U_t^\psi$}\label{ssec:Ut}
We now prove some mapping properties for $U_t^\psi$, where $\psi$ is a global solution 
of \eqref{HE} as in Theorem \ref{thm:HE-exist}.
Our first lemma shows that $U_t^\psi$ is bounded as an operator in $H^s$.

\begin{lemma}\label{lemHs}
Under the conditions of Theorem \ref{thm:HE-exist}, there exists an absolute constant $C>0$ such that
\begin{align}\label{lemHsconc}
{\big\| U_t^\psi \big\|}_{H^s\mapsto H^s} \leq C .
\end{align}
\end{lemma}

In order to prove Proposition \ref{prop:ic-ext21}, we also need to estimate the norm 
of $U_t^\psi$ as an operator from $H^\gamma\cap L^2_\gamma$ to $L^2_\gamma$. Note that here we need to impose stronger regularity conditions on $v$ than in the previous lemma. 

\begin{lemma}\label{lemL2gamma}
Under the conditions of Theorem \ref{thm:max-vel-HE0} and with $\gamma < d/q-1$, 
for all $\varphi \in H^\gamma\cap L^2_\gamma$, we have
\begin{align}\label{lemL2conc}
\big\| U_t^\psi \varphi \big\|_{L^2_\gamma} \lesssim \langle t\rangle^{\gamma} \|\varphi\|_{H^\gamma}
  +\|\varphi\|_{L^2_\gamma}.
\end{align}
\end{lemma} 

We mention that our proof of Lemma \ref{lemL2gamma} actually establishes a stronger result than \eqref{lemL2conc}, as we estimate $\|\langle H\rangle^{\frac{\ell}{2}} \langle x\rangle^{\gamma'} U_t^\psi \varphi \| $ for suitable values of $\ell$ and $\gamma'$, see \eqref{eq:induc3}. The estimate \eqref{lemL2conc} is however sufficient for our purpose.

\subsection{Estimates on differences}\label{ssec:diff}

The result of this subsection are needed to prove the contraction property 
in the fixed point argument used to establish Proposition \ref{prop:ic-ext21}. 
The first lemma estimates the differences between two solutions of \eqref{HE}.

\begin{lemma}\label{lemLpdiff}
Under the conditions of Theorem \ref{thm:HE-exist}, consider $\psi_t$ and $\varphi_t$ two global solutions 
of \eqref{HE} as in Theorem \ref{thm:HE-exist}. We have
\begin{align}\label{lemLpdiffconc}
{\big\| \psi_t - \varphi_t \big\|}_{L^p} \lesssim \jt^{-d/(2q)} {\| \psi_0 - \varphi_0 \|}_{L^{p'} \cap H^s}. 
\end{align}
\end{lemma}

Using Lemma \ref{lemLpdiff}, it is not difficult to prove the following lemma.

\begin{lemma}\label{lemWtdiff}
Under the conditions of Theorem \ref{thm:HE-exist}, consider $\psi_t$ and $\varphi_t$ two global solutions 
of \eqref{HE} as in Theorem \ref{thm:HE-exist}. We have 
\begin{align}\label{lemWrdiffconc0}
{\big\| W_t^\psi - W_t^\varphi \big\|}_{L^\infty} \lesssim \varepsilon \jt^{-d/q} {\| \psi_0 - \varphi_0 \|}_{L^{p'} \cap H^s},
\end{align}
\begin{align}\label{lemWrdiffconc0a}
{\big\| W_t^\psi - W_t^\varphi \big\|}_{W^{s,2q}} \lesssim  \jt^{-d/(2q)} 
  {\| \psi_0 - \varphi_0 \|}_{L^{p'} \cap H^s},
\end{align}
and
\begin{align}\label{lemWrdiffconc}
{\big\| W_t^\psi - W_t^\varphi \big\|}_{H^s\mapsto H^s} \lesssim  \jt^{-d/(2q)} {\| \psi_0 - \varphi_0 \|}_{L^{p'} \cap H^s}. 
\end{align}
If in addition $v$ satisfies \eqref{v-cond}, then 
\begin{align}\label{lemWrdiffconcb}
{\big\| W_t^\psi - W_t^\varphi \big\|}_{W^{\gamma,\infty}} 
  \lesssim \varepsilon \jt^{-d/q} {\| \psi_0 - \varphi_0 \|}_{L^{p'} \cap H^s},
\end{align}
and
\begin{align}\label{lemWrdiffconc2}
{\big\| W_t^\psi - W_t^\varphi \big\|}_{H^\gamma\mapsto H^\gamma} 
  \lesssim \varepsilon \jt^{-d/q} {\| \psi_0 - \varphi_0 \|}_{L^{p'} \cap H^s}. 
\end{align}
\end{lemma}

Finally, we estimate the norms of the differences of the flows $U_t^\psi-U_t^\varphi$.

\begin{lemma}\label{lemHsdiff}
Under the conditions of Theorem \ref{thm:HE-exist}, consider $\psi_t$ and $\varphi_t$ two global solutions 
of \eqref{HE} as in Theorem \ref{thm:HE-exist}.
Then we have
\begin{align}\label{lemHsdiffconc}
{\big\| U_t^\psi - U_t^\varphi \big\|}_{H^s \mapsto H^s} 
  \lesssim  {\| \psi_0 - \varphi_0 \|}_{L^{p'} \cap H^s}. 
\end{align} 
\end{lemma}

\begin{lemma}\label{lemL2diff}
Under the conditions of Theorem \ref{thm:max-vel-HE0} and with $\gamma\le d/q-1$, consider $\psi_t$ and $\varphi_t$ two global solutions 
of \eqref{HE} as in Theorem \ref{thm:HE-exist}.
For all $f \in H^\gamma\cap L^2_\gamma$, we have
\begin{align}\label{lemL2diffconc}
{\big\| \big(U_t^\psi - U_t^\varphi \big)f\big\|}_{L^2_\gamma}
  \lesssim {\| \psi_0 - \varphi_0 \|}_{L^{p'} \cap H^\gamma} \big ( \langle t\rangle^\gamma  \|f\|_{H^\gamma} + \|f\|_{L^2_\gamma} \big). 
\end{align} 
\end{lemma}

\medskip
\section{Proof of Proposition \ref{prop:ic-ext21}}\label{sec:propic}

In the first part of this proof, we will 
omit the superindex $\psi$ for $W_s^\psi$, $U_{s}^\psi$ and $g_{+}^{\psi}$ and so on,
when there is no risk of confusion. 

Let $0<\varepsilon\ll1$ and $\phi\in \mathcal{B}_\gamma^s(\varepsilon)$.  
Fix $\tilde{\varepsilon}>0$ such that $C_0 \varepsilon \leq \tilde{\varepsilon}\ll1$ for some absolute $C_0>1$ to be determined. 
We will show that the map 
\begin{equation*}
\psi_0\mapsto F_\phi(\psi_0):=g_+^{\Psi (\psi_0)} (H)\phi
\end{equation*}
is a contraction in $\mathcal{B}_\gamma^s(\tilde{\varepsilon})$. 
Let $\psi_0\in\mathcal{B}_\gamma^s(\tilde{\varepsilon})$. From \eqref{g+expl} we have
\begin{align} \label{gt21}
F_\phi(\psi_0)=g(H)\phi - i\int_0^\infty U_r^{-1} W_r' \, U_r  \phi \, dr, \quad W_r' := [g(H), W_r]. 
\end{align}

\smallskip
With $p$ as in \eqref{pqp1}, by H\"older's inequality we have 
$\| f \|_{L^{p'}} \leq \| \langle x \rangle^{-\gamma} \|_{L^{2q}} \| \langle x \rangle^{\gamma} f \|_{L^2}$,
and, since $\gamma > d/(2q)$, $L^2_\gamma \subset L^{p'}$.
Therefore, using Theorem \ref{thm:HE-exist}, for any given $\psi_0 \in \mathcal{B}_\gamma^s(\tilde{\varepsilon})$ 
we can construct 
a unique global solution $\psi_t$ to \eqref{HE} satisfying \eqref{Hs-ests}-\eqref{pq-ests}.

\def\jnab{\lan \nabla \ran}
\def\jH{\lan H \ran}

\medskip
{\it Boundedness on $H^s$.}
We begin by proving the bound on $H^s$. We want to show
\begin{equation}\label{eq:g(H)-est0}
\big\| g_+(H) \phi \big\|_{H^s} \ls \| \phi\|_{H^s}.
\end{equation}
Since \eqref{eq:g(H)-est0} obviously holds true with $g$ instead of $g_+$,
by \eqref{gt21} it suffices to prove that
\begin{equation}\label{eq:int00}
\big \| U_{r}^{-1} \,  g(H) W_r  U_{r} \big\|_{H^s} \lesssim \langle r \rangle^{-d/(2q)}\|\phi\|_{H^s},
\end{equation}
and use that $d/(2q) > 1$.
Note that we are writing $g(H)W_r$ instead of the full commutator $W'_r$ from \eqref{gt21}.
We will adopt a similar convention in the rest of the proofs in this section.
The estimate \eqref{eq:int00} exchanging the position of $g(H)$ and $W_r$ can be obtained in the same way
(since, in particular, we will not make use of the fact that the projection $g(H)$ is bounded from $H^s$ to $L^2$
in what follows).

By Lemma \ref{lemHs}, we have ${\big\| U_t^\psi \big\|}_{H^s\mapsto H^s} \lesssim1$, while Lemma \ref{lem:Wt2} gives
$\big\|W_r^\psi\big\|_{H^s\mapsto H^s} \lesssim \tilde\varepsilon \langle r \rangle^{-d/(2q)}$. 
This implies \eqref{eq:int00}.

\medskip
{\it Boundedness on $L^2_\gamma$}.
Next, we prove boundedness on $L^2_\gamma$,
that is,
\begin{equation}\label{eq:g(H)-est}
\big\|  
   g_+(H) \phi \big\|_{L^2_\gamma} \ls \|\phi\|_{L^2_\gamma\cap H^\gamma}, \qquad \phi \in L^2_\gamma \cap H^\gamma.
\end{equation}
Since it is not difficult to show the necessary estimate
for $g(H)$, by \eqref{gt21} it suffices to prove that, for any $\phi \in L^2_\gamma \cap H^\gamma$, we have
\begin{equation}\label{eq:int}
\big\| U_r^{-1} g(H)W_r \, U_r  \, \phi \, \big\|_{L^2_\gamma} \ls \lan r \ran^{\gamma - d/q} \|\phi\|_{L^2_\gamma\cap H^\gamma},
\end{equation}
and then use $\gamma < d/q-1$ so that the above bound is integrable.
Note that we are once again just working with $g(H)W_r$ instead of the commutator.

First, using Lemma \ref{lemL2gamma}, we obtain
\begin{equation}\label{eq:int_a}
\big\| U_r^{-1} g(H)W_r \, U_r  \, \phi \, \big\|_{L^2_\gamma} \ls \lan r \ran^\gamma \big \| g(H)W_r  U_r  \phi  \big\|_{H^\gamma}+\big \| g(H)W_r U_r  \phi  \big\|_{L^2_\gamma}.
\end{equation}
For the first term, we use Lemmas \ref{lem:Wt2} and \ref{lemHs}, which yield
\begin{equation}\label{eq:int_b}
\big \| g(H)W_r  U_r  \phi  \big\|_{H^\gamma}\lesssim \tilde\varepsilon^2 \langle r\rangle^{-d/q}\|\phi\|_{H^\gamma}.
\end{equation}
For the second term, since $g(H):L^2_\gamma\to L^2_\gamma$ is bounded, we can write
\begin{align}
\big \| g(H)W_r U_r  \phi  \big\|_{L^2_\gamma} &\lesssim \|W_r\|_{L^\infty} \big \| U_r  \phi  \big\|_{L^2_\gamma} \notag \\
&\lesssim \tilde\varepsilon^2\langle r\rangle^{-d/q} \big( \langle r\rangle^\gamma\|\phi\|_{H^\gamma}+\|\phi\|_{L^2_\gamma}\big), \label{eq:int_c}
\end{align}
having used Lemmas \ref{lem:Wt} and \ref{lemL2gamma} in the second inequality. Equations \eqref{eq:int_a}, \eqref{eq:int_b} and \eqref{eq:int_c} imply \eqref{eq:int} and therefore \eqref{eq:g(H)-est}.

\medskip
{\it Contraction.}
To prove that $F_\phi$ is contractive, we use arguments that are similar to those above,
but we now need to apply them to the difference
$F_\phi(\varphi_0) - F_\phi(\psi_0)$, for data $\psi_0, \varphi_0 \in L^2_\gamma\cap H^s$.
Let us denote by $\psi_t$ and $\varphi_t$ the respective global solutions guaranteed by Theorem \ref{thm:HE-exist}.

We skip the estimate for the Sobolev norm since it is easier, and concentrate on estimating the $L^2_\g$ norm.
We restore the superindex $\psi$ for  $H_t^\psi, W_t^\psi$, $U_t^\psi$ and so on.
For $\psi_0, \varphi_0 \in \mathcal{B}^s_\gamma(\tilde \varepsilon) \subset L^2_\gamma\cap H^s$, and $\phi \in \mathcal{B}^s_\gamma(\varepsilon)$,
we estimate first
\begin{align}
\label{contmain0}
& \big \|F_\phi(\varphi_0) - F_\phi(\psi_0) \big \|_{L^2_\gamma} \leq D_1 + D_2 + D_3 + \mbox{`similar'},
  \\
\label{contmain01}
& D_1 := \int_0^\infty \big \| (U_r^\varphi)^{-1} \, g(H) W_{r}^{\varphi} \,
	\big(U_r^\varphi - U_r^\psi \big) 
	\phi \big\|_{L^2_\gamma} \, dr,
\\
\label{contmain02}
& D_2 := \int_0^\infty \big \|  (U_r^\varphi)^{-1} g(H) 
  \big(W_{r}^{\varphi}- W_{r}^{\psi} \big) 
  U_r^\psi 
  \phi \big\|_{L^2_\gamma} \, dr,
\\
\label{contmain03}
& D_3 := \int_0^\infty \big \|  \big( (U_r^\varphi)^{-1} - (U_r^\psi)^{-1} \big) g(H) W_{r}^{\psi} 
	U_r \phi \big\|_{L^2_\gamma} \, dr,
\end{align}
where we are again only looking at terms with $g(H) W_r$ and can disregard the `similar' ones with $W_r g(H)$.
We then want to prove
\begin{align}\label{contmain1}
D_1, \, D_2, \, D_3 \lesssim \varepsilon \, \tilde{\varepsilon} \| \psi_0 - \varphi_0 \|_{L^2_\gamma \cap H^s}.
\end{align}

The terms $D_1$ and $D_3$ can be estimated similarly so we just focus on the first. Using Lemma \ref{lemL2gamma}, we obtain
\begin{align}
&\big \| (U_r^\varphi)^{-1} \, g(H) W_{r}^{\varphi} \big(U_r^\varphi - U_r^\psi \big) \phi \big\|_{L^2_\gamma} \notag \\
&\ls \lan r \ran^\gamma \big \| g(H) W_{r}^{\varphi} \big(U_r^\varphi - U_r^\psi \big) \phi  \big\|_{H^\gamma}+\big \| g(H) W_{r}^{\varphi} \big(U_r^\varphi - U_r^\psi \big) \phi  \big\|_{L^2_\gamma}. \label{eq:int_a1}
\end{align} 
To estimate the first term in the rhs of \eqref{eq:int_a1}, we use Lemmas \ref{lem:Wt2} and \ref{lemHsdiff} yielding
\begin{align}
 \big \| g(H) W_{r}^{\varphi} \big(U_r^\varphi - U_r^\psi \big) \phi  \big\|_{H^\gamma} \lesssim \tilde\varepsilon^2 \langle r\rangle^{-d/q} \|\psi_0-\varphi_0\|_{L^{p'}\cap H^\gamma} \|\phi\|_{H^\gamma}. \label{eq:int_a2}
\end{align} 
Since $g(H):L^2_\gamma\to L^2_\gamma$ is bounded, the second term in the rhs of \eqref{eq:int_a1} can be estimated by
\begin{align}
&\big \| g(H) W_{r}^{\varphi} \big(U_r^\varphi - U_r^\psi \big) \phi  \big\|_{L^2_\gamma} \notag \\
&\lesssim \| W_{r}^{\varphi} \|_{L^\infty} \big \| \big(U_r^\varphi - U_r^\psi \big) \phi  \big\|_{L^2_\gamma} \notag \\
&\lesssim \langle r\rangle^{-d/q} \tilde\varepsilon^2{\| \psi_0 - \varphi_0 \|}_{L^{p'} \cap H^\gamma} \big ( \langle r\rangle^\gamma  \|\phi\|_{H^\gamma} + \|\phi\|_{L^2_\gamma} \big), \label{eq:int_a3}
\end{align} 
the second inequality being a consequence of Lemmas \ref{lem:Wt} and \ref{lemL2diff}. 
Inserting \eqref{eq:int_a2} and \eqref{eq:int_a3} into \eqref{eq:int_a1} gives
\begin{align}
&\big \| (U_r^\varphi)^{-1} \, g(H) W_{r}^{\varphi} \big(U_r^\varphi - U_r^\psi \big) \phi \big\|_{L^2_\gamma} \notag \\
&\lesssim \langle r\rangle^{-d/q+\gamma} \tilde\varepsilon^2{\| \psi_0 - \varphi_0 \|}_{L^{p'} \cap H^\gamma} \|\phi\|_{L^2_\gamma\cap H^\gamma} . \label{eq:int_a4}
\end{align} 
Therefore, since $\gamma<d/q-1$ and $\|\phi\|_{L^2_\gamma\cap H^\gamma}\le\varepsilon$, we have shown that
\begin{align}\label{contmain1a}
D_1 \lesssim \varepsilon \tilde{\varepsilon}^2 \| \psi_0 - \varphi_0 \|_{L^2_\gamma \cap H^\gamma}.
\end{align}
The same bound holds for $D_3$. 

To estimate $D_2$, we write using Lemma \ref{lemL2gamma}
\begin{align}
&\big \|  (U_r^\varphi)^{-1} g(H) \big(W_{r}^{\varphi}- W_{r}^{\psi} \big) U_r^\psi \phi \big\|_{L^2_\gamma}  \notag \\
&\lesssim \langle r\rangle^\gamma \big \| \big(W_{r}^{\varphi}- W_{r}^{\psi} \big) U_r^\psi \phi \big\|_{H^\gamma} +\big \| \big(W_{r}^{\varphi}- W_{r}^{\psi} \big) U_r^\psi \phi \big\|_{L^2_\gamma}. \label{boundD2}
\end{align}
The first term is estimated using Lemma \ref{lemWtdiff}, which gives
\begin{align}
\big \| \big(W_{r}^{\varphi}- W_{r}^{\psi} \big) U_r^\psi \phi \big\|_{H^\gamma} &\lesssim \tilde\varepsilon \langle r \rangle^{-d/q} {\| \psi_0 - \varphi_0 \|}_{L^{p'} \cap H^\gamma} \big\| U_r^\psi \phi \big\|_{H^\gamma} \notag \\
&\lesssim \tilde\varepsilon \langle r \rangle^{-d/q} {\| \psi_0 - \varphi_0 \|}_{L^{p'} \cap H^\gamma} \| \phi \|_{H^\gamma}, \label{boundD2a}
\end{align}
the second inequality following from Lemma \ref{lemHs}. The second term in the rhs of \eqref{boundD2} is estimated as
\begin{align}
&\big \| \big(W_{r}^{\varphi}- W_{r}^{\psi} \big) U_r^\psi \phi \big\|_{L^2_\gamma} \notag \\
&\lesssim \big\|W_r^\varphi-W_r^\psi\big\|_{L^\infty} \big\| U_r^\psi \phi \big\|_{L^2_\gamma} \notag \\
&\lesssim \tilde\varepsilon \langle r \rangle^{-d/q} {\| \psi_0 - \varphi_0 \|}_{L^{p'} \cap H^\gamma} \big( \langle r\rangle^\gamma\| \phi \|_{H^\gamma}+\|\phi\|_{L^2_\gamma}\big), \label{boundD2b}
\end{align}
where we have used Lemmas \ref{lemWtdiff} and \ref{lemL2gamma}  to obtain the second inequality. Plugging \eqref{boundD2a} and \eqref{boundD2b} into \eqref{boundD2} gives
\begin{align}
&\big \|  (U_r^\varphi)^{-1} g(H) \big(W_{r}^{\varphi}- W_{r}^{\psi} \big) U_r^\psi \phi \big\|_{L^2_\gamma} \lesssim \varepsilon \tilde\varepsilon \langle r \rangle^{-d/q} {\| \psi_0 - \varphi_0 \|}_{L^{p'} \cap H^\gamma} ,
\end{align}
since $\| \phi \|_{L^2_\gamma\cap H^\gamma}\le\varepsilon$, and therefore
\begin{align}\label{contmain1b}
D_2 \lesssim \varepsilon \tilde{\varepsilon} \| \psi_0 - \varphi_0 \|_{L^2_\gamma \cap H^\gamma}.
\end{align}

Hence, since $s\geq \gamma$, we have proven \eqref{contmain1}, which implies
\begin{align}\label{contmain1'}
\big \| F_\phi(\varphi_0) - F_\phi(\psi_0) \big \|_{L^2_\gamma} &
\lesssim \tilde{\varepsilon} \varepsilon \big \| \varphi_0 - \psi_0 \big\|_{L^2_\gamma\cap H^s}.
\end{align}

The analogous estimate for the Sobolev norm, that is,
\begin{align}\label{contmain2}
\big \| F_\phi(\varphi_0) - F_\phi(\psi_0) \big \|_{H^s} & 
  \lesssim \tilde{\varepsilon} \varepsilon \big \| \varphi_0 - \psi_0 \big \|_{L^2_\gamma\cap H^s},
\end{align}
can be obtained similarly, and is in fact easier to show. 
Since $\tilde{\varepsilon}\ll1$, \eqref{contmain1'}-\eqref{contmain2} imply that $F_\phi$ is a contraction.

\medskip
{\it Injectivity.}
Finally, we verify that the map $\phi \mapsto \psi_0 = g_+^\psi(H) \phi$ is injective on $g(H)\mathcal{B}^s_\gamma(\varepsilon)$. 
Indeed, assume that for $\ell=1,2$ we have $\phi_\ell \in \mathcal{B}^s_\gamma(\varepsilon)$ with $g(H)\phi_1 \neq g(H)\phi_2$, 
and let  $\psi_{0,\ell} \in \mathcal{B}^s_\gamma(C_0\varepsilon)$ 
be the (unique) solutions of
$\psi_{0,\ell} = g_+^{\psi_\ell}(H) \phi_\ell$ with $\psi_\ell = U_t^{\psi_\ell}\psi_{0,\ell}$.
Then, from \eqref{gt21} and arguments similar to those above, we can estimate 
\begin{align*}
& {\big\| g_+^{\psi_1}(H) \phi_1 - g_+^{\psi_2}(H) \phi_2 \big\|}_{H^s} 
\\ & \geq {\big\| g(H) (\phi_1-\phi_2) \big\|}_{H^s} - C\varepsilon^2 
  \big[ {\| \phi_1-\phi_2\|}_{H^s} + {\| \psi_1-\psi_2\|}_{H^s \cap L^2_\gamma} \big].
\end{align*}
This concludes the proof. $\hfill \Box$

\medskip
\appendix

\section{Proof of Proposition \ref{prop:condV}}\label{app:condV}

\begin{proof}[Proof of Proposition \ref{prop:condV}]
Let $V=V_+-V_-$ be such that \eqref{Aas2-1} holds.
We first prove that $-\frac12\Delta+V_+$ does not have nonpositive eigenvalues nor a resonance at $0$. Clearly, since $-\frac12\Delta+V_+\ge0$, its spectrum is contained in $\R_+$. We show that $0$ is not an eingenvalue nor a resonance of $-\frac12\Delta+V_+$.

Suppose that $\phi\in \cap_{\gamma>\frac12} L^2_{-\gamma}$ is a solution to $(-\frac12\Delta+V_+)\phi=0$.  
Since $V_+^{\frac12}(-\Delta)^{-1}V_+^{\frac12}$ is a bounded operator 
in $L^2$ by the assumption \eqref{Aas2-1}, this implies that 
\begin{equation*}
V_+^{\frac12}\phi = -2V_+^{\frac12}(-\Delta)^{-1}V_+\phi.
\end{equation*}
Taking the scalar product with $V_+^{\frac12}\phi$ gives
\begin{equation*}
\big\|V_+^\frac12\phi\big\|^2_{L^2} = -2\langle\phi,V_+(-\Delta)^{-1}V_+\phi\rangle_{L^2}.
\end{equation*}
Since in addition $V_+^{\frac12}(-\Delta)^{-1}V_+^{\frac12}$ is nonnegative, this shows that $V_+^{\frac12}\phi=0$. Hence $(-\frac12\Delta)\phi=0$, which implies that $\phi=0$. Thus $-\frac12\Delta+V_+$ does not have nonpositive eigenvalues nor a resonance at $0$.

Next we show that $-\frac12\Delta+V$ does not have negative eigenvalues. Let $\lambda>0$. Let $\phi\in L^2$ be such that $(-\frac12\Delta+V+\lambda)\phi=0$. As above, since $(-\frac12\Delta+V_++\lambda)$ is invertible, this implies that
\begin{equation}\label{BS1}
\big\|V_-^\frac12\phi\big\|^2_{L^2} = \langle\phi,V_-(-\frac12\Delta+V_++\lambda)^{-1}V_-\phi\rangle_{L^2}.
\end{equation}
We have $-\frac12\Delta+V_++\lambda\ge-\frac12\Delta+\lambda$ and hence, since both operators are invertible and $-\frac12\Delta+\lambda$ is positive,
\begin{equation}\label{eq:s1}
\big(-\frac12\Delta+V_++\lambda\big)^{-1}\le\big(-\frac12\Delta+\lambda\big)^{-1}.
\end{equation}
Indeed, for any self-adjoint, invertible operators $A$ and $B$ having the same domain and satisfying $A\le B$ and $A>0$, we have $B^{-1/2}AB^{-1/2}\le 1$, which implies $B^{1/2}A^{-1}B^{1/2}\ge 1$ giving $A^{-1}\ge B^{-1}$. Relation \eqref{eq:s1} together with \eqref{BS1}, implies
\begin{equation}\label{BS-est1}
\big\|V_-^\frac12\phi\big\|^2_{L^2} \le \langle\phi,V_-(-\frac12\Delta+\lambda)^{-1}V_-\phi\rangle_{L^2}.
\end{equation}
Due to the assumption $\langle x\rangle^{\alpha}V_-(x)\le\delta$, this yields
\begin{align}\label{BS-est2_ab}
\big\|V_-^\frac12\phi\big\|^2_{L^2} &\le 
  \big\langle \langle x\rangle^{\frac{\alpha}{2}} V_-\phi , \langle x\rangle^{-\frac{\alpha}{2}} 
  (-\frac12\Delta+\lambda)^{-1} \langle x\rangle^{-\frac{\alpha}{2}}\langle x\rangle^{\frac{\alpha}{2}}
  V_-\phi\big\rangle_{L^2}
  \\
& \le\delta \big\|V_-^{\frac12}\phi\big\|^2_{L^2} \big\|\langle x\rangle^{-\frac{\alpha}{2}} 
  (-\frac12\Delta+\lambda)^{-1}\langle x\rangle^{-\frac{\alpha}{2}} \big\|.
\end{align}

Since $\alpha>2$, the operator $\langle x\rangle^{-\frac{\alpha}{2}} 
(-\frac12\Delta+\lambda)^{-1}\langle x\rangle^{-\frac{\alpha}{2}}:L^2\to L^2$ 
is bounded uniformly in $\lambda\ge0$. Hence, for $\delta$ small enough, 
we deduce that $V_-\phi=0$. Therefore $(-\frac12\Delta+V_++\lambda)\phi=0$ 
which yields $\phi=0$ since we know that $-\frac12\Delta+V_+$ does not have negative eigenvalues.

Now we show that $0$ is neither an eigenvalue nor a resonance of $-\frac12\Delta+V$. Suppose that $\phi\in \cap_{\gamma>\frac12} L^2_{-\gamma}$ is a solution to $(-\frac12\Delta+V)\phi=0$. Letting $\lambda\to0$ in \eqref{eq:s1} shows that $(-\frac12\Delta+V_+)^{-1}:L^2_\gamma\to L^2_{-\gamma}$ is bounded for $\gamma>1$. Hence, using \eqref{Aas2-1}, we see that the equation $(-\frac12\Delta+V)\phi=0$ implies $V_-^{\frac12}\phi=V_-^{\frac12}(-\frac12\Delta+V_+)^{-1}V_-\phi$. Taking the scalar product with $V_-^{\frac12}\phi$ and using $V_-(x)\le\delta\langle x\rangle^{-\alpha}$ we obtain
\begin{align*}
\big\|V_-^\frac12\phi\big\|^2_{L^2} &\le\delta^2\big\|V_-^\frac12\phi\big\|^2_{L^2} \big\|\langle x\rangle^{-\frac{\alpha}{2}} (-\frac12\Delta+V_+)^{-1}\langle x\rangle^{-\frac{\alpha}{2}} \big\|.
\end{align*}
For $\delta$ small enough, we can conclude that $V_-\phi=0$. Therefore $(-\frac12\Delta+V_+)\phi=0$, which yields that $\phi=0$ since we know that $0$ is not a resonance of $-\frac12\Delta+V_+$.

It remains to prove that $-\frac12\Delta+V$ doest not have positive eigenvalues. This is a standard result given the condition \eqref{Aas2-1} (see e.g. \cite[Theorem XIII.58]{RS4}).
\end{proof}

\section{Proof of Theorem \ref{thm:HE-exist}}\label{sec:HE-exist}

To prove Theorem \ref{thm:HE-exist}, we use  
 standard arguments, combining energy and decay estimates. Recall that global existence in $L^2(\mathbb{R}^d)$ of solutions 
satisfying \eqref{L2-ests} is standard (see e.g. \cite{Caz}). 
Local existence in $H^s(\mathbb{R}^d)$ is also standard, see e.g. \cite[Theorem 4.10.1]{Caz}, for $s>d/(2q)$, an integer. 
We prove it here for convenience of the reader as the proof under our conditions 
is simpler than that of \cite{Caz} which is done for fairly general nonlinearities. 
We 
then bootstrap it to the global existence. 
We also use some of the estimates, or variants of them, in Section \ref{ssec:use}.

As is standard in the local existence proofs, 
we use the Duhamel principle to rewrite the Hartree equation \eqref{HE} as a fixed point problem 
\begin{align}\label{HEfp}
 \psi_{t}= G_{t}(\psi), \quad G_{t}(\psi):=e^{-iHt}\psi_{0} -i  \int_{0}^{t} e^{-iH(t-r)}W^{\psi}_{r}\psi_{r}  dr,
\end{align}
(recall the definition of the nonlinear potential $W^\psi_t$ from \eqref{defW_t})
and then use the contraction mapping principle to prove the existence of a unique fixed point in a ball in $H^{s}$.

By time-reversal symmetry we may assume $t\geq 0$.
Elementary estimates under Condition \eqref{Aas1} show
 the equivalence of the norms $\| \psi \|_{H^s}$ and $\big \| (H+C)^{s/2}\psi \big \|_{L^2}$, where $C\ge -\inf H+1$, which   
  yields the bound
\begin{align}\label{propagHsbound}
& {\big\| e^{itH} f \big\|}_{H^s} \lesssim {\| f \|}_{H^s}.
\end{align}

Using definition \eqref{HEfp} of $G$ and estimate \eqref{propagHsbound}, we find right away
 for all $t \in [0,T]$:
\begin{align}\label{fp-bound'}
\big \| G_{t}(\psi) \big \|_{H^s}& \leq \big \| \psi_{0} \big \|_{H^s}  +
	 \big \| W^{\psi}_{t}\psi_{t} \big\|_{L^{1}_{t}([0,T])H^s_{x}}.
\end{align}
Applying the  Kato-Ponce inequality (or fractional Leibniz rule) and the weak Young's inequality,
recalling that $p=2q'=2q/(q-1)$, and observing that  
\begin{align}\label{W-estpq}
1/2 = 1/(2q) + 1/p ,  \quad 1+ 1/(2q) = 1/q + 1/p_{1},
\end{align}
where $1/p_{1} = 1/2 + 1/p $, we estimate the $H^{s}$-norms of the last term in \eqref{fp-bound} 
for fixed $t$ as follows 
\begin{align}
\nonumber
\big \|  W^{\psi }_{t}  \psi_{t}\big \|_{H^s} &\lesssim \big\| W^{\psi}_{t}  \big\|_{L^\infty} \big\| \psi_t \big\|_{H^s} + \big\| W^{\psi}_{t}  \big\|_{W^{s,2q}} \big\| \psi_t \big\|_{L^p} 
\\
\nonumber 
& \lesssim \big\| v \big\|_{L^{q,\infty}} \Big( \big\|  |\psi_t|^2 \big\|_{L^{q'}} \big\|\psi_{t}  \big\|_{H^{s}}
	+ \big\|  |\psi_t|^2 \big\|_{ W^{s,p_{1}} }\big \| \psi_t \big \|_{L^{p}} \Big)  \\
\label{W-est}
& \lesssim  \big\| v \big\|_{L^{q,\infty}} \big\| \psi_{t}  \big\|_{L^{p}}^{2} \big\|\psi_{t} \big\|_{H^{s}}.
\end{align}
Now, consider the Banach spaces $ H^{s}_{T} :=L^{\infty}\big ([0,T],H^{s}\big )$ 
and $L^{p}_{T}:=L^{\infty}\big ([0,T],L^{p}\big )$, 
with the norms $\|f\|_{H^{s}_T}:=\sup_{0\le t\le T} \|f(t)\|_{H^{s}}$ 
and $\|f\|_{L^{p}_T}:=\sup_{0\le t\le T} \|f(t)\|_{L^{p}} $, and let  $\psi_{t} \in H^{s}_{T} $ 
such that $\psi_{0} \in  H^s$. 
Then the last two inequalities give, after taking the supremum in $t$ over $[0,T]$,  
\begin{align}\label{fp-bound}
\big\| G_{t}(\psi)\big \|_{H^s_{T}} \lesssim  \big \| \psi_{0} \big \|_{H^s} 
  +  T\big\| v \big\|_{L^{q,\infty}} \big\| \psi_{t}  \big\|_{L^{p}_{T}}^{2} \big\|\psi_{t}  \big\|_{H^s_{T}}.
\end{align}
Hence, since $H^s\hookrightarrow L^{p}$ (as $s>d/(2q)=d(1/2-1/p)$), 
the map $G$ takes the ball $H^s_{T, R}$ in $H^s_{T}$ of the radius $R$ 
centred at the origin into itself, provided $R$ satisfies 
$R\ge C(\big \| \psi_{0} \big \|_{H^s} +  T\big\| v \big\|_{L^{q,\infty}} R^3)$ with $C$ large enough.

Similarly, we estimate the difference $\big\| G_{t}(\psi)  - G_{t}(\phi) \big \|_{H^s_{T}}$:
\[
\big\| G_{t}(\psi)  - G_{t}(\phi) \big \|_{H^s_{T}} \lesssim  
 T\big\| v \big\|_{L^{q,\infty}} \big( \big\| \psi_{t}  \big\|_{H^s_{T}} + \big\| \phi_{t}  
 \big\|_{H^s_{T}}\big)^{2} \big\|\psi_{t}-\phi_{t}  \big\|_{H^s_{T}}.
\]
Hence $G$ is a contraction on $H^s_{T, R}$ provided $R$ and $T$ satisfy 
$R\ge C(\big \| \psi_{0} \big \|_{H^s} +  T\big\| v \big\|_{L^{q,\infty}} R^3)$ 
and $C T\big\| v \big\|_{L^{q,\infty}} R^{2}<1$
for some constant $C>1$ independent of $R$ and $T$. 
This implies local well-posedness in $H^s_{T}$, provided the local time of existence $T>0$ is sufficiently small.

The local existence proven above implies that the bounds in \eqref{Hs-ests} and \eqref{pq-ests}
hold for some finite time. We now bootstrap the  local existence and these bounds to the global existence 
and the global bounds. 

More precisely, we assume, for some $T>0$ and $D$ large enough, 
that the solution $\psi_t$ of \eqref{HE},
with $\|\psi_0\|_{L^{p'} \cap H^s}\le \veps$, satisfies 
\begin{align}  
\label{Hs-estsboot0}
& \sup_{t\in[0,T]} \|\psi_t\|_{H^s} \leq 2D \|\psi_0\|_{H^s}, 
\\
\label{pq-estsboot0}
& \sup_{t\in[0,T]} \big( \langle t \rangle^{d/2q}  \|\psi_t\|_{L^p} \big) \leq 2D\|\psi_0\|_{L^{p'} \cap H^s},
\end{align}
and then show that
\begin{align}  
\label{Hs-estsboot}
& \sup_{t\in[0,T]} \|\psi_t\|_{H^s} \leq D \|\psi_0\|_{H^s}, 
\\
\label{pq-estsboot}
& \sup_{t\in[0,T]} \big(   \langle t \rangle^{d/2q}  \|\psi_t\|_{L^p} \big) \leq D\|\psi_0\|_{L^{p'} \cap H^s}.
\end{align}

To begin with, we mention first that under Conditions \eqref{Aas1} and \eqref{Aas2'}, 
the unitary evolution of the linear part $e^{-iHt}$ of \eqref{HE} 
is bounded from $L^{p'}$ to $L^{p}$ and satisfies the dispersive estimate 
\begin{equation}\label{eq:disp}
\|e^{-iHt} f\|_{L^p} \lesssim t^{-d(\frac{1}{2}-\frac{1}{p})} \|f\|_{L^{p'}} , \quad t >0 ,
\end{equation}
(see for example \cite{Yajima, BeSch}, the introduction of \cite{GHW} and the recent survey \cite{SchSurvey}). 
This estimate, together with  estimate \eqref{propagHsbound} and with the Sobolev embedding 
$H^s\hookrightarrow L^{p}$ (as $s>d(1/2-1/p)$) and the relation 
$1/2-1/p=1/(2q)$ yield the bound
\begin{equation}\label{disp-est-comb}
\|e^{-iHt} f\|_{L^p} \lesssim \lan t\ran^{-d/(2q)} \|f\|_{H^s\cap L^{p'}}. 
\end{equation}

Now, applying estimates \eqref{fp-bound}, 
 \eqref{Hs-estsboot0} and \eqref{pq-estsboot0} to the fixed point equation \eqref{HEfp} gives 
\begin{align*}
\big \| \psi_{t} \big\|_{H^{s}} &\le C\big \|\psi_{0} \big \|_{H^{s}} 
  + 4C D^{3} \big \| v \big \|_{L^{q,\infty}}\big \|\psi_{0} \big \|_{L^{p'} \cap H^{s}}^{2}
  \|\psi_{0} \big \|_{H^{s}} \int_{0}^{t}\langle r \rangle^{-d/q} \; dr 
  \\
&  \le C\varepsilon + D^{3}\tilde{C}\varepsilon^{3}, 
\end{align*}
where the integral converges since $d/q>2$. 
Altogether, this implies \eqref{Hs-estsboot}, provided $C + D^{3}\tilde{C}\varepsilon^{2}\le D$,  
for $D$ sufficiently large. This bound implies also that $\psi_{t} \in C([0, T],H^{s})$.

Let us now prove \eqref{pq-estsboot}.
Applying the $L^p$-norm to the fixed point equation \eqref{HEfp} 
and using estimate \eqref{disp-est-comb}, we obtain, for all $t\in[0,T]$,
\begin{align}\label{Lp-est}
\notag\| \psi_{t} \|_{L^p} &\le \big \| e^{-iHt}\psi_{0} \big\|_{L^p} 
  + \Big \| \int_{0}^{t} e^{-iH(t-r)}W^{\psi}_{r}\psi_{r} dr \Big \|_{L^p}\\
& \le \lan t\ran^{-d/(2q)}\big \| \psi_{0} \big\|_{H^s\cap L^{p'}} 
  +  \int_{0}^{t} \lan t-r\ran^{-d/(2q)}\big \|W^{\psi}_{r}\psi_{r} \big \|_{H^s\cap L^{p'}} dr.\end{align}
Now, observing that $1/p'=1/q+1/p$, using the H\"older estimate $\big \|  W^{\psi }_{t}  \psi_{t}\big \|_{L^{p'}}
  \lesssim  \big\| W^{\psi }_{t} \big\|_{L^{q}} \big\| \psi_{t}  \big\|_{L^{p}}$ 
and then the weak Young one $\big\| W^{\psi }_{t} \big\|_{L^{q}}\ls  
\big\| v \big\|_{L^{q,\infty}} \big\| |\psi_{t}|^2  \big\|_{L^{1}}$, 
together with \eqref{W-est}, \eqref{Lp-est}, \eqref{Hs-estsboot0}, \eqref{pq-estsboot0} and $\|\psi_0\|_{L^{p'} \cap H^s}\le \varepsilon$, gives
\begin{align*}
\| \psi_t \|_{L^p} 
  & \ls \langle t \rangle^{-d/2q} \varepsilon  + \varepsilon^3 D^{3}  
  \int_0^{t} \langle t - r \rangle^{-d/2q}  \langle r \rangle^{-d/2q} dr.
\end{align*}
Since $d/2q > 1$, the integral above is bounded by $C\langle t \rangle^{-d/2q}$ and hence \eqref{pq-estsboot} follows provided we choose $\varepsilon$ so that $1   + \varepsilon^2 D^{3}\ll D$. 

Thus, we have shown \eqref{Hs-estsboot} and \eqref{pq-estsboot} which allows us to iterate the local existence result 
by a standard continuation argument to complete the proof of the theorem. 
$\hfill \Box$

\medskip
\section{Proof of Lemmas 
 \ref{lem:Wt_ab}--\ref{lemL2diff}}\label{ssec:use}

In this section we prove the results stated in Section \ref{sec:mapping} which were used in Section \ref{sec:propic}.
 Some of the arguments used below are similar to those in the proof of Theorem \ref{thm:HE-exist} just given above, so we will skip some details.

\medskip
{\it Notation}: 
As above, we will use in this section $U_t^\psi$ (and similarly $U_t^\varphi$) to denote the flow of 
$H_t^\psi = H + f(|\psi|^2) = -\tfrac{1}{2}\Delta + V + v\ast |\psi|^2$, see \eqref{defW_t},
where $V$ satisfies the conditions  
\eqref{Aas1}-\eqref{Aas2'} 
and $\psi = \psi_t$ is the unique global $H^s$ ($s>d/(2q)$) solution of \eqref{HE}
under the conditions of Theorem \ref{thm:HE-exist};
in particular, we are assuming that the initial data $\psi_0$ satisfies \eqref{ic},
and $\psi_t$ satisfies \eqref{Hs-ests} and \eqref{pq-ests}.
Also, the indexes $p$, $q$ and $p_1$ satisfy the same relations used so far:
\begin{align}\label{relpq}
1/p = 1/2- 1/(2q) 
 \qquad \text{and}  \quad 1/p_1=1/2+1/p\end{align}
($p=\infty$ for $q=1$). 
Recall that we write $L^{1,\infty}\equiv L^1$, $W^{\gamma,(q,\infty)}\equiv W^{\gamma,q}$ in the case where $q=1$. We also recall that Condition \eqref{Aas1} implies the equivalence of the norms $\| \psi \|_{H^s}$ and $\big \| (H+C)^{s/2}\psi \big \|_{L^2}$, where $C\ge -\inf H+1$.
\medskip

\begin{proof}[Proof of Lemma \ref{lem:Wt_ab}]
Using Young's and H{\"o}lder's inequalities, recalling that $1/p_1=1-1/(2q)=1/2+1/p$, we obtain
\begin{equation*}
\big\|W_t^\psi\big\|_{W^{s,2q}}\lesssim\|v\|_{L^{q,\infty}}\big\||\psi_t|^2\big\|_{W^{s,p_1}}
  \lesssim\|v\|_{L^{q,\infty}}\|\psi_t\|_{L^p}\|\psi_t\|_{H^s}.
\end{equation*}
Hence \eqref{eq:Wt-decay_ab} follows from \eqref{Hs-ests}--\eqref{pq-ests} in Theorem \ref{thm:HE-exist}. 
\end{proof}

\begin{proof}[Proof of Lemma \ref{lem:Wt2}]
Let $f\in H^s$, with $s>d(1/2-1/p)=d/(2q)$. Applying the  Kato-Ponce inequality (or fractional Leibniz rule), 
we have
\begin{equation}
\big\|W_t^\psi f\big\|_{H^s}\lesssim \big\| W^{\psi}_{t}  \big\|_{L^\infty} \big\| f\big\|_{H^s} + \big\| W^{\psi}_{t}  \big\|_{W^{s,2q}} \big\| f \big\|_{L^p} .\label{eq:z1}
\end{equation}
Using Sobolev's embedding $H^s(\R^d)\hookrightarrow L^p(\R^d)$, for  $s>
 d/(2q)$, together with Lemmas \ref{lem:Wt} and \ref{lem:Wt_ab}, 
 we obtain \eqref{eq:Wt_Hs}. To prove \eqref{eq:Wt_Hs2}, we write similarly,
\begin{equation*}
\big\|W_t^\psi f\big\|_{H^\gamma}\lesssim \big\| W^{\psi}_{t}  \big\|_{L^\infty} \big\| f\big\|_{H^\gamma} 
  + \big\| W^{\psi}_{t}  \big\|_{W^{\gamma,\infty}} \big\| f \big\|_{L^2} ,
\end{equation*}
and hence the result follows from \eqref{eq:Wt-decay}-\eqref{eq:Wt-decay2} 
of Lemma \ref{lem:Wt} applied to the last two inequalities.
\end{proof}

\begin{proof}[Proof of Lemma \ref{lem:Wt3}]
Using Young's inequality, we write
\begin{align}
\big\|W_t^\psi\big\|_{W^{\sigma,\infty}}&\lesssim\big\|\big(\langle\nabla\rangle^{\sigma_1}v\big)*\big(\langle\nabla\rangle^{\sigma_2} |\psi_t|^2\big)\big\|_{L^\infty}\notag\\
&\lesssim\big\|\langle\nabla\rangle^{\sigma_1}v\big\|_{L^{q,\infty}}\big\|\langle\nabla\rangle^{\sigma_2} \psi_t\big\|_{L^{2q'}}\|\psi_t\|_{L^{2q'}}\notag\\
&=\big\|\langle\nabla\rangle^{\sigma_1}v\big\|_{L^{q,\infty}}\big\|\psi_t\big\|_{W^{\sigma_2,p}}\|\psi_t\|_{L^{p}},\label{eq:h2}
\end{align}
since $p=2q'$. Next, the Gagliardo-Nirenberg-Sobolev inequality gives
\begin{equation}
\|\psi_t\|_{W^{\sigma_2,p}}\lesssim\|\psi_t\|_{H^s}^{1-\beta}\|\psi_t\|_{L^p}^\beta, \label{eq:h3}
\end{equation}
provided that
\begin{equation}\label{eq:h1}
\frac1p=\frac{\sigma_2}{d}+\Big(\frac12-\frac{s}{d}\Big)(1-\beta)+\frac{\beta}{p}, \qquad 0<\beta<1.
\end{equation}
Now, given $\varepsilon'>0$ such that $d/(2q)-\varepsilon'>0$, we choose $\beta$ such that $\frac{d\beta}{2q}=\frac{d}{2q}-\varepsilon'$. The condition \eqref{eq:h1} then yields
\begin{equation*}
s=\frac{d}{2q}+\frac{\sigma_2}{1-\beta}>\frac{d}{2q}+\frac{2d}{d-2q}.
\end{equation*}
Equations \eqref{eq:h2}, \eqref{eq:h3} together with Theorem \ref{thm:HE-exist} imply
\begin{align*}
\big\|W_t^\psi\big\|_{W^{\sigma,\infty}}\lesssim\big\|\langle\nabla\rangle^{\sigma_1}v\big\|_{L^{q,\infty}}\|\psi_t\|^{1+\beta}_{L^{p}}\lesssim\big\|\langle\nabla\rangle^{\sigma_1}v\big\|_{L^{q,\infty}}\langle t\rangle^{(1+\beta)\frac{d}{2q}}.
\end{align*}
This proves the lemma.
\end{proof}

\begin{proof}[Proof of Lemma \ref{lemHs}]
For any $f_0 \in H^s$ we let $f = f_t := U_t^\psi f_0$ and write
\begin{align}\label{HEfp'}
f = e^{-iHt} f_{0} -i\int_{0}^{t} e^{-iH(t-r)}W^{\psi}_{r}f \, dr.
\end{align}
For the integrated term, using $1/p=1/2-1/(2q)$, we estimate
\begin{align}
{\Big\|  \int_{0}^{t} e^{-iH(t-r)}W^{\psi}_{r}f \, dr \Big\|}_{H^s} 
  & \lesssim \int_0^t \Big ( {\big\| W^{\psi}_{r} \big\|}_{L^\infty} {\| f \|}_{H^s} 
  + {\big\| W^{\psi}_{r} \big\|}_{W^{s,2q}} {\| f \|}_{L^p} \Big)
  \, dr  \notag\\
  &\lesssim \varepsilon \int_0^t \langle r\rangle^{-\frac{d}{2q}} \|f\|_{H^s} \, dr,  \label{prHs1}
\end{align}
having used Sobolev embedding and Lemmas \ref{lem:Wt} and \ref{lem:Wt_ab}.
Using \eqref{prHs1} and $\| e^{-iHt}f_0 \|_{H^s} \lesssim \|f_0\|_{H^s}$ (see \eqref{propagHsbound})
in \eqref{HEfp'} we can then obtain \eqref{lemHsconc} by Gronwall's inequality since $d/(2q)>1$. 
\end{proof}

\begin{proof}[Proof of Lemma \ref{lemL2gamma}]

Let $n$ be a nonnegative integer such that $n \le \gamma$ and suppose that $v\in W^{n,(q,\infty)}$.
We first prove by induction that for all $k\in\{0,\dots,n\}$ and $\ell \in \{ 0 , \dots, n-k\}$,
\begin{align}\label{eq:induc}
{\big\|\langle H\rangle^{\frac{\ell}{2}} \langle x\rangle^kU_t^\psi \varphi \big\|} 
	\le C \sum_{j=0}^k \langle t\rangle^j\big\| \langle H\rangle^{\frac{\ell+j}{2}} \langle x\rangle^{k-j}\varphi\big\|. 
\tag{$\mathcal{H}_{k,\ell}$}
\end{align}
Note that \eqref{eq:induc} is a natural statement in view of the 
dispersion relation and the consequent localization property ``$|x|^2 \approx t^2 H$'' for (linear) Schr\"odinger flows.

For $k=0$, $(\mathcal{H}_{0,\ell})$ holds for any $\ell\in \{0,\dots,n\}$ 
as follows from Lemma \ref{lemHs}. Let $k\in\{0,\dots,n-1\}$. 
Suppose that $(\mathcal{H}_{k',\ell})$ holds for all $k'\le k$ and all $\ell \in\{0\dots,n-k'\}$.
First we show that $(\mathcal{H}_{k+1,0})$. Using the relation
\begin{align}\label{comm-id1-app}
[A, U_r] 
=- iU_r \int_0^r  U_{\tau}^{-1}[A,H_{\tau}] U_{\tau} \, d\tau,
\end{align}
we write
\begin{align}
&\big\| \langle x\rangle^{k+1}U_t^\psi \varphi \big\|\notag\\
&\lesssim \big\| \langle x\rangle^{k+1} \varphi \big\| + \int_0^t  \big \|  [\langle x\rangle^{k+1},\Delta] U_{r}^\psi \varphi \big \| \, dr \notag\\
&\lesssim \big\| \langle x\rangle^{k+1} \varphi \big\| + \int_0^t \big( \big \| \langle \nabla\rangle \langle x\rangle^{k}U_{r}^\psi \varphi \big \|+\big\|\langle x\rangle^{k-1} U_{r}^\psi \varphi \big \|\big) \, dr \notag\\
&\lesssim \big\| \langle x\rangle^{k+1} \varphi \big\| + \int_0^t  \big( \big \| \langle H\rangle^{\frac12} \langle x\rangle^{k} U_{r}^\psi \varphi \big \| +\big\|\langle x\rangle^{k-1}U_{r}^\psi \varphi \big \| \big) \, dr \notag
\\
\notag
&\lesssim \big\| \langle x\rangle^{k+1} \varphi \big\| +  \sum_{j=0}^k \int_0^t \langle r\rangle^j\big\| 
\langle H\rangle^{\frac{1+j}{2}}
\langle x\rangle^{k-j}\varphi\big\| \, dr
\\
& \lesssim
\big\| \langle x\rangle^{k+1} \varphi \big\| +  \sum_{j=0}^k \langle r\rangle^{j+1}
\big\| \langle H\rangle^{\frac{1+j}{2}} \langle x\rangle^{k-j}\varphi\big\|, 
\label{eq:g1}
\end{align}
where we used the induction hypothesis in the inequality before last. 
This easily implies that $(\mathcal{H}_{k+1,0})$ holds. Next, let $\ell\in\{0,\dots,n-(k+2)\}$. 
Assuming in addition that $(\mathcal{H}_{k+1,\ell'})$ holds for all $\ell'\le\ell$, we show that $(\mathcal{H}_{k+1,\ell+1})$ holds. As above, we write
\begin{align}
&\big\|\langle H\rangle^{\frac{\ell+1}{2}} \langle x\rangle^{k+1}U_t^\psi \varphi \big\| \notag\\
&\lesssim \big\|\langle H\rangle^{\frac{\ell+1}{2}} \langle x\rangle^{k+1} \varphi \big\|
+ \int_0^t  \big \| [\langle H\rangle^{\frac{\ell+1}{2}}\langle x\rangle^{k+1},H_{\tau}^\psi] U_{r}^\psi \varphi \big \| \, dr \notag\\
&\lesssim \big\|\langle H\rangle^{\frac{\ell+1}{2}} \langle x\rangle^{k+1} \varphi \big\|
+ \int_0^t  \big \| \langle H\rangle^{\frac{\ell+1}{2}} [\langle x\rangle^{k+1},\Delta] U_{r}^\psi \varphi \big \| \, dr \notag\\
&\quad+ \int_0^t \big \| [ \langle H\rangle^{\frac{\ell+1}{2}} , W_r^\psi ] \langle x\rangle^{k+1} U_{r}^\psi \varphi \big \| \, dr. \label{eq:t1}
\end{align}
For the first integrated term, one verifies that
\begin{align}
&\big \| \langle H\rangle^{\frac{\ell+1}{2}} [\langle x\rangle^{k+1},\Delta] U_{r}^\psi \varphi \big \| \notag\\
&\qquad  \qquad \lesssim \big \| \langle H\rangle^{\frac{\ell+2}{2}} \langle x\rangle^{k} U_{r}^\psi \varphi \big \|+\big\| \langle H\rangle^{\frac{\ell+1}{2}}\langle x\rangle^{k-1} U_{r}^\psi \varphi \big \|. \label{eq:t2}
\end{align}
The second one can be estimated by
\begin{align}
\big \| [ \langle H\rangle^{\frac{\ell+1}{2}} , W_r^\psi ] \langle x\rangle^{k+1} U_{r}^\psi \varphi \big \| &\lesssim \sum_{\ell'=0}^\ell\big\|W_r^\psi\big\|_{W^{\ell-\ell'+1,\infty}}\big\|\langle H\rangle^{\frac{\ell'}{2}}\langle x\rangle^{k+1} U_{r}^\psi \varphi \big \|\notag \\
&\lesssim  \langle r\rangle^{-\frac{d}{q}} \sum_{\ell'=0}^\ell\big\|\langle H\rangle^{\frac{\ell'}{2}}\langle x\rangle^{k+1} U_{r}^\psi \varphi \big \|,\label{eq:t3}
\end{align}
where the last inequality follows from Lemma \ref{lem:Wt}. Inserting \eqref{eq:t2} and \eqref{eq:t3} into \eqref{eq:t1} and using the induction hypothesis, we obtain that
\begin{align}
&\big\|\langle H\rangle^{\frac{\ell+1}{2}} \langle x\rangle^{k+1}U_t^\psi \varphi \big\| \notag\\
&\qquad\lesssim \big\|\langle H\rangle^{\frac{\ell+1}{2}} \langle x\rangle^{k+1} \varphi \big\|
+ \sum_{j=0}^k \int_0^t \langle r\rangle^j\big\|\langle H\rangle^{\frac{n-k+j}{2}}\langle x\rangle^{k-j}\varphi\big\| \, dr \notag\\
&\qquad\qquad+ \sum_{j=0}^k \int_0^t \langle r\rangle^{-\frac{d}{q}}\langle r\rangle^j\big\|\langle H\rangle^{\frac{n-k+j}{2}}\langle x\rangle^{k+1-j}\varphi\big\| \, dr. \label{eq:t4}
\end{align}
Since $\langle r\rangle^{-d/q+j}$ is integrable for all $j\le n$ (since $n \le \gamma < d/q-1$), 
one deduces from the previous estimate that $(\mathcal{H}_{k+1,\ell+1})$ holds.

Thus, we have proven that \eqref{eq:induc} holds for all $k\in\{0,\dots,n\}$ 
and $\ell \in \{ 0 , \dots , n-k\}$.

Next, we claim that, for all $\ell \in \{ 0 , \dots , n-k\}$ we have
\begin{align}\label{eq:induc2}
\big\|\langle H\rangle^{\frac{\ell}{2}} \langle x\rangle^kU_t^\psi \varphi \big\| 
	\lesssim \langle t\rangle^k \big\|\langle H\rangle^{\frac{k+\ell}{2}}\varphi\big\|
	+ \|\langle H\rangle^{\frac{\ell}{2}}\langle x\rangle^k\varphi\|,
\end{align}
for all nonnegative integers $\ell,k$ such that $v\in W^{\ell+k,(q,\infty)}$.
It suffices to bound each term in the sum on the right-hand side of \eqref{eq:induc} by the
right-hand side of \eqref{eq:induc2} (which is equivalent to the first and last terms in the sum in \eqref{eq:induc});
that is, it suffices to prove that
for all $j=0,\dots k$
\begin{align}\label{eq:induc2'}
\langle t\rangle^j\big\| \langle H\rangle^{\frac{\ell+j}{2}} \langle x\rangle^{k-j}\varphi\big\|
  \ls \langle t\rangle^k \big\|\langle H\rangle^{\frac{k+\ell}{2}}\varphi\big\|
  + \|\langle H\rangle^{\frac{\ell}{2}}\langle x\rangle^k\varphi\|.
\end{align}
Recall that, under our assumptions on $V$,
we have the equivalence of the Sobolev norms 
\begin{align}\label{Sobequiv}
{\| \langle H \rangle^{s/2} f \|}_{L^2} \approx {\| \langle \nabla \rangle^s f \|}_{L^2} = {\| f \|}_{H^s}.
\end{align}
Then, we 
square \eqref{eq:induc2'} to see that it is equivalent to
\begin{align}\label{eq:induc2''}
\langle t\rangle^{2j} {\big\| \langle x\rangle^{k-j}\varphi\big\|}_{H^{\ell+j}}^2
	 \ls \langle t\rangle^{2k} {\big\| \varphi \big\|}_{H^{\ell+k}}^2
	+ {\| \langle x\rangle^k\varphi\|}_{H^{\ell}}^2, \qquad 0\leq j\leq k.
\end{align}
On the standard $H^s$ space we can now 
use the (inhomogeneous) Littlewood-Paley decomposition $f = \sum_{N\geq 0} P_N f$, 
with $\widehat{P_N f}(\xi) := \chi_N(\xi) \widehat{f}(\xi)$
where $\chi_N$ is a bump function supported on $|\xi| \in [2^{N-1}, 2^{N+1}]$ for $N>0$,
and compactly supported in $[-2,2]$ for $N=0$.
Recall that 
\begin{align}\label{HsLP}
{\| f \|}_{H^s}^2 \approx \sum_{N\geq 0} 2^{2Ns} {\| P_N f \|}^2_{L^2} ;
\end{align}
moreover, by commuting $P_N$ and $\jx$, we can see that
\begin{align}\label{HsLPx}
{\| P_N \langle x \rangle f \|}_{L^2} \approx {\| \langle x \rangle P_N f \|}_{L^2},
\end{align}
having slightly abused notation by disregarding similar terms 
with $P_{N-1}$ and  $P_{N+1}$ instead of $P_N$ on the right-hand side. 
Then, we write
\begin{align} 
\nonumber
\langle t\rangle^{2j} {\big\| \langle x\rangle^{k-j}\varphi \big\|}_{H^{\ell+j}}^2
  & \approx \langle t\rangle^{2j} \sum_{N\geq 0} 2^{2N(\ell+j)} {\|  P_N \langle x\rangle^{k-j}\varphi \|}^2_{L^2}
  \\
  & \lesssim \langle t\rangle^{2j} \sum_{N\geq 0} 2^{2N(\ell+j)} {\| \langle x\rangle^{k-j} P_N \varphi \|}^2_{L^2}.
\label{eq:induc2.1}
\end{align}
In \eqref{eq:induc2.1} we then distinguish the cases $\jx \leq \jt 2^N$ and $\jx > \jt 2^N$.
When $\jx \leq \jt 2^N$ we bound
\begin{align*}
\begin{split}
&  \langle t\rangle^{2j} \sum_{N\geq 0} 2^{2N(\ell+j)} {\| \langle x\rangle^{k-j} P_N \varphi \|}^2_{L^2(\jx \leq \jt 2^N)}
  \\
  & \ls  \langle t\rangle^{2j} \sum_{N\geq 0}  2^{2N(\ell+j)} {\| ( \jt 2^N)^{k-j} P_N \varphi \|}^2_{L^2}
  \\
  & =  \langle t\rangle^{2k} \sum_{N\geq 0} 2^{2N(\ell+k)} {\| P_N \varphi \|}^2_{L^2} 
  \approx  \langle t\rangle^{2k} {\| \varphi \|}_{H^{\ell+k}}^2;
\end{split}
\end{align*}
this last term is accounted for in the right-hand side of \eqref{eq:induc2''}.
Similarly, we bound the contribution from \eqref{eq:induc2.1} in the region $\jx > \jt 2^N$ by
\begin{align*}
\begin{split}
& \langle t\rangle^{2j} \sum_{N\geq 0} 2^{2N(\ell+j)} {\| \langle x\rangle^{k-j} P_N \varphi \|}^2_{L^2(\jx > \jt 2^N)}
\\
& \ls  \sum_{N\geq 0} 2^{2N\ell} {\| \langle x\rangle^{k}( \langle x \rangle^{-1} 
  \langle t\rangle 2^{N})^j P_N \varphi \|}^2_{L^2(\jx > \jt 2^N)}
  \\
  & = \sum_{N\geq 0} 2^{2N\ell} {\| \langle x\rangle^{k} P_N \varphi \|}^2_{L^2}
  \approx {\| \langle x \rangle^k \varphi \|}_{H^{\ell}}^2,
\end{split}
\end{align*}
having used \eqref{HsLPx} and \eqref{HsLP} for the last equivalence. 
This gives us \eqref{eq:induc2''} and therefore \eqref{eq:induc2}.

Now, let $\lfloor \gamma\rfloor$ be the integer part of $\gamma \geq 1$. 
We claim that, by interpolation, one can deduce from \eqref{eq:induc2} that, 
for all $\gamma'\ge0$ and $\ell$ an integer such that $\gamma'+\ell\le\lfloor \gamma\rfloor$,
\begin{align}\label{eq:induc3}
\big\|\langle H\rangle^{\frac{\ell}{2}} \langle x\rangle^{\gamma'} U_t^\psi \varphi \big\| 
\lesssim \langle t\rangle^{\gamma'} \big\|\langle H\rangle^{\frac{\gamma'+\ell}{2}}\varphi\big\|
+\|\langle H\rangle^{\frac{\ell}{2}}\langle x\rangle^{\gamma'}\varphi\|.
\end{align}
To see this, consider the linear operator $T_\ell: \varphi \rightarrow \langle H \rangle^{\ell/2} U_t^\psi \varphi$
for $\ell = 0,\dots n-k$ and $v \in W^{\ell+k,(q,\infty)}$ as above.
Then, using the equivalence of Sobolev norms \eqref{Sobequiv} and commuting (standard) derivatives and weights,
the inequality \eqref{eq:induc2} says that $T$
maps $H^k(\langle t \rangle^k dx) \cap L^2_k$ into $L^2_k$ (recall the definition above \eqref{Bsg}).
Standard interpolation between Sobolev spaces and between weighted $L^2$ spaces then gives that
$T$ maps $H^{\gamma'}(\langle t \rangle^{\gamma'} dx) \cap L^2_{\gamma'}$ into $L^2_{\gamma'}$
that is, using again the equivalence \eqref{Sobequiv}, inequality \eqref{eq:induc3}.

Taking $\ell=0$ in \eqref{eq:induc3}, we obtain
\begin{align}\label{eq:induc4}
\big\|\langle x\rangle^{\gamma'} U_t^\psi \varphi \big\| 
\lesssim \langle t\rangle^{\gamma'} \big\|\langle H\rangle^{\frac{\gamma'}{2}}\varphi\big\|
+\|\langle x\rangle^{\gamma'}\varphi\|,
\end{align}
for all $\gamma'\le \lfloor \gamma\rfloor$. 
We then write, similarly as in \eqref{eq:g1},
\begin{align}
& \big\| \langle x\rangle^{\gamma}U_t^\psi \varphi \big\|\notag
\\
& \lesssim \big\| \langle x\rangle^{\gamma} \varphi \big\| 
  + \int_0^t  \big( \big \| \langle H\rangle^{\frac12} \langle x\rangle^{\gamma-1} 
  U_{r}^\psi \varphi \big \| +\big\|\langle x\rangle^{\gamma-2}U_{r}^\psi \varphi \big \| \big) \, dr . 
\label{eq:ha1}
\end{align}
Since $\gamma-1\le \lfloor \gamma\rfloor$, we can apply \eqref{eq:induc3}, which yields
\begin{equation*}
\big \| \langle H\rangle^{\frac12} \langle x\rangle^{\gamma-1} U_{r}^\psi \varphi \big \|
\lesssim \langle r\rangle^{\gamma-1} \big\|\langle H\rangle^{\frac{\gamma}{2}}\varphi\big\|
+\|\langle H\rangle^{\frac{1}{2}}\langle x\rangle^{\gamma-1}\varphi\|.
\end{equation*}
Inserting this into \eqref{eq:ha1}, and then applying \eqref{eq:induc3} again (if $\gamma > 2$) 
gives
\begin{align}
\nonumber
& \big\| \langle x\rangle^{\gamma}U_t^\psi \varphi \big\|
\\
\nonumber
& \lesssim \big\| \langle x\rangle^{\gamma} \varphi \big\| 
  + \langle t\rangle^{\gamma} \big\|\langle H\rangle^{\frac{\gamma}{2}}\varphi\big\|
  + \langle t\rangle \|\langle H\rangle^{\frac{1}{2}}\langle x\rangle^{\gamma-1}\varphi\|
  + \int_0^t \big\|\langle x\rangle^{\gamma-2}U_{r}^\psi \varphi \big \| \, dr  
\\
& \lesssim \big\| \langle x\rangle^{\gamma} \varphi \big\| 
  + \langle t\rangle^{\gamma} \big\|\langle H\rangle^{\frac{\gamma}{2}}\varphi\big\|
  + \langle t\rangle \|\langle H\rangle^{\frac{1}{2}}\langle x\rangle^{\gamma-1}\varphi\|.
  \label{eq:induc40}
\end{align}
Finally, we recall \eqref{eq:induc2'} which, for $\ell=0,j=1$ reads
\begin{align*}
\langle t\rangle \| \langle H \rangle^{\frac{1}{2}} \langle x \rangle^{k-1}\varphi \big\|
	 \ls \langle t\rangle^k \| \langle H \rangle^{\frac{k}{2}} \varphi \big\|
	+ \| \langle x\rangle^k \varphi\|,
\end{align*}
and using interpolation for the weighted Sobolev spaces 
we obtain the same inequality above with $\gamma$ replacing $k$;
plugging it into \eqref{eq:induc40} we get
\begin{align}\label{eq:induc41}
\big\|\langle x\rangle^{\gamma} U_t^\psi \varphi \big\| 
\lesssim \langle t\rangle^{\gamma} \big\|\langle H\rangle^{\frac{\gamma}{2}}\varphi\big\|
+\|\langle x\rangle^{\gamma}\varphi\|.
\end{align}
This concludes the proof of the lemma.
\end{proof}

\begin{proof}[Proof of Lemma \ref{lemLpdiff}]
The proof of \eqref{lemLpdiffconc} uses similar estimates to \eqref{Lp-est} applied to the difference of two solutions. 
From Duhamel's representation we have
\begin{align*}
u_t = e^{-iHt} u_{0} -i \int_{0}^{t} e^{-iH(t-r)}W^{u}_{r} u_r \, dr, \qquad u \in \{\psi,\varphi\}.
\end{align*}
A bound by the right-hand side of \eqref{lemLpdiffconc} for the difference of the linear flows,
$e^{-iHt} \psi_{0} - e^{-iHt} \varphi_{0}$, follows directly from the linear dispersive estimate \eqref{disp-est-comb}.
For the difference of the nonlinear parts we have
\begin{align}
\nonumber
& {\Big\| \int_{0}^{t} e^{-iH(t-r)} \big( W^{\psi}_{r} \psi_r - W^{\varphi}_{r} \varphi_r \big)\, dr \Big\|}_{L^p} 
\\
\label{prLpdiff1}
&  \lesssim  \int_{0}^{t} \langle t-r \rangle^{-d/(2q)} \Big( 
  {\big\| (W^{\psi}_{r} - W^{\varphi}_{r})\psi_r \big\|}_{L^{p'}\cap H^s} 
  + {\big\| W^{\varphi}_{r} (\psi_r - \varphi_r) \big\|}_{L^{p'}\cap H^s} \Big)
  \, dr .
\end{align}
The estimates for the $H^s$ norms are the same used for the contraction argument in the 
proof of Theorem \ref{thm:HE-exist} 
so we only show the bound for the $L^{p'}$ norm.

To estimate the first term in \eqref{prLpdiff1} we note 
 that, in view of \eqref{relpq}, $1/p' = 1/2 + 1/(2q)$ and $1+1/(2q) = 1/q + 1/p_1$, with $1/p_1 = 1/2 + 1/p$. Then,
using H\"older's and Young's inequalities, and \eqref{L2-ests}, we get
\begin{align}
{\big\| (W^{\psi}_{r} - W^{\varphi}_{r})\psi_r \big\|}_{L^{p'}} 
  & \lesssim {\big\| v \ast \big( |\psi_{r}|^2 - |\varphi_r|^2 ) \big\|}_{L^{2q}} {\| \psi_r \|}_{L^2} \notag
  \\
& \lesssim {\big\| |\psi_{r}|^2 - |\varphi_r|^2 \big\|}_{L^{p_1}} {\| \psi_r \|}_{L^2} \notag
  \\
  & \lesssim \varepsilon^2 {\| \psi_{r} - \varphi_{r} \|}_{L^p}. \label{prLpdiff2}
\end{align}

For the second term in \eqref{prLpdiff1}, using again H\"older with \eqref{relpq}, Young's inequality
and \eqref{L2-ests}, we have
\begin{align}\label{prLpdiff3}
\begin{split}
{\big\| W^{\varphi}_{r} (\psi_r - \varphi_r) \big\|}_{L^{p'}}
  & \lesssim  {\| \psi_r - \varphi_r \|}_{L^{p}} {\big\| v \ast |\varphi_{r}|^2 \big\|}_{L^{q}}
  \lesssim \varepsilon^2 {\| \psi_{r} - \varphi_{r} \|}_{L^p}.
\end{split}
\end{align}

Plugging \eqref{prLpdiff2} and \eqref{prLpdiff3} into \eqref{prLpdiff1} we have obtained 
\begin{align*}
{\| \psi_t - \varphi_t \|}_{L^p} \leq C \jt^{-d/(2q)} {\| \psi_0 - \varphi_0 \|}_{L^{p'}} 
 + C \varepsilon^2 \int_0^t \langle t-r \rangle^{-d/(2q)}  {\| \psi_{r} - \varphi_{r} \|}_{L^p} \, dr.
\end{align*}
From this we can obtain the conclusion \eqref{lemLpdiffconc} by 
a bootstrap argument (such as the one for the quantity in \eqref{pq-estsboot0} in the proof of Theorem \ref{thm:HE-exist}).
\end{proof}

\begin{proof}[Proof of Lemma \ref{lemWtdiff}]
\eqref{lemWrdiffconc0} follows from the H{\"o}lder and Young inequalities, \eqref{pq-ests} and Lemma \ref{lemLpdiff}:
\begin{align*}
\big\| W^{\psi}_{t} -W_t^\varphi \big\|_{L^\infty}&\lesssim\|v\|_{L^{q,\infty}} \big\| |\psi_t|^2-|\varphi_t|^2 \big\|_{L^{q'}}\\
&\lesssim \varepsilon \|v\|_{L^{q,\infty}} \langle t\rangle^{-d/(2q)} \big\| \psi_t-\varphi_t  \big\|_{L^{p}}\\
&\lesssim \varepsilon \|v\|_{L^{q,\infty}} \langle t\rangle^{-d/q} \big\| \psi_0-\varphi_0  \big\|_{L^{p'}\cap H^s}.
\end{align*}
Similarly, using \eqref{Hs-ests},
\begin{align*}
\big\| W^{\psi}_{t} -W_t^\varphi \big\|_{W^{s,2q}}&\lesssim\|v\|_{L^{q,\infty}} \big\| |\psi_t|^2-|\varphi_t|^2 \big\|_{W^{s,p_1}}\\
&\lesssim \|v\|_{L^{q,\infty}} \big\| \psi_t-\varphi_t  \big\|_{L^{p}}\\
&\lesssim \|v\|_{L^{q,\infty}} \langle t\rangle^{-d/(2q)} \big\| \psi_0-\varphi_0  \big\|_{L^{p'}\cap H^s}.
\end{align*}
This proves \eqref{lemWrdiffconc0a}. Assuming in addition that $v$ satisfies \eqref{v-cond} with $\gamma\le s$, \eqref{lemWrdiffconcb} can be proven in the same way.

To prove \eqref{lemWrdiffconc}, we proceed similarly as in the proof of Lemmas \ref{lem:Wt} and \ref{lem:Wt2}, 
writing for $f\in H^s(\R^d)$,
\begin{align*}
&\big\|(W_t^\psi-W_t^\varphi) f\big\|_{H^s} \\
&\lesssim \big\| W^{\psi}_{t} -W_t^\varphi \big\|_{L^\infty} \big\| f\big\|_{H^s} 
+ \big\| W^{\psi}_{t}-W_t^\varphi \big\|_{W^{s,2q}} \big\| f \big\|_{L^p}
\\
& \lesssim \varepsilon 
\langle t\rangle^{-d/q} 
\big\| \psi_0-\varphi_0  \big\|_{L^{p'}\cap H^s} \big\| f\big\|_{H^s} + 
  \langle t\rangle^{-d/(2q)} \big\| \psi_0-\varphi_0 \big\|_{L^{p'}\cap H^s} \big\| f \big\|_{L^p},
\end{align*}
where we used \eqref{lemWrdiffconc0} and \eqref{lemWrdiffconc0a} in the last inequality. 
Together with Sobolev's embedding and Lemma \ref{lemLpdiff}, this proves \eqref{lemWrdiffconc}. 
If in addition $v$ satisfies \eqref{v-cond} with $\gamma\le s$, then, using \eqref{lemWrdiffconc0}
and \eqref{lemWrdiffconcb} gives
\begin{align*}
&\big\|(W_t^\psi-W_t^\varphi) f\big\|_{H^\gamma} 
\\
&\lesssim \big\| W^{\psi}_{t} -W_t^\varphi \big\|_{L^\infty} \big\| f\big\|_{H^\gamma} 
+ \big\| W^{\psi}_{t}-W_t^\varphi \big\|_{W^{\gamma,\infty}} \big\| f \big\|_{L^2}
\\
&\lesssim \varepsilon \|v\|_{W^{\gamma,(q,\infty)}} \langle t\rangle^{-d/q} \big\| \psi_0-\varphi_0  \big\|_{L^{p'}\cap H^s} \big\| f\big\|_{H^\gamma} .
\end{align*}
This establishes \eqref{lemWrdiffconc2}.
\end{proof}

\begin{proof}[Proof of Lemma \ref{lemHsdiff}]
For a given $\phi \in H^s$, ${\| \phi \|}_{H^s} \leq 1$, let us define
\begin{align*}
(u_1)_t = u_1 := U_t^\psi \phi, \qquad (u_2)_t = u_2 := U_t^\varphi \phi.
\end{align*}
We want to estimate ${\|u_1-u_2 \|}_{H^s} \lesssim \varepsilon {\| \psi_0 - \varphi_0 \|}_{L^{p'} \cap H^s}$. 
From Duhamel's formula we have
\begin{align}
\nonumber
{\| u_1-u_2 \|}_{H^s} & \lesssim
  \int_{0}^{t} {\big\| W^{\psi}_{r}(u_1)_r - W^{\varphi}_{r}(u_2)_r \big\|}_{H^s} 
  \, dr 
\\
& \lesssim \int_{0}^{t} \Big( 
  {\big\| (W^{\psi}_{r} - W^{\varphi}_{r}) (u_1)_r \big\|}_{H^s} 
  + {\big\| W^{\varphi}_{r} \big( (u_1)_r - (u_2)_r \big) \big\|}_{H^s} \Big)
  \, dr. 
\label{prHsdiff1} 
\end{align}

For the first of the two quantities in the integral in \eqref{prHsdiff1} we use Lemma \ref{lemWtdiff}, which gives
\begin{align}
 {\big\| (W^{\psi}_{r} - W^{\varphi}_{r})(u_1)_r \big\|}_{H^s} 
  & \lesssim 
 \langle r \rangle^{-d/(2q)} {\| \psi_0 - \varphi_0 \|}_{L^{p'} \cap H^s} {\| (u_1)_r \|}_{H^s} \notag \\
  & \lesssim \langle r \rangle^{-d/(2q)} {\| \psi_0 - \varphi_0 \|}_{L^{p'} \cap H^s}. \label{prHsdiff10}
\end{align}
having also used \eqref{lemHsconc}.

For the second term in the integral in \eqref{prHsdiff1}, we use Lemma \ref{lem:Wt2}, which gives
\begin{align}
& {\big\| W^{\varphi}_{r} ((u_1)_r-(u_2)_r) \big\|}_{H^s}  
  \lesssim \varepsilon \langle r \rangle^{-d/(2q)} {\| (u_1)_r-(u_2)_r \|}_{H^s}. \label{prHsdiff3} 
\end{align}

Putting together \eqref{prHsdiff1}-\eqref{prHsdiff3} we have obtained
\begin{align*}
{\| (u_1)_t- (u_2)_t \|}_{H^s} & \lesssim {\| \psi_0 - \varphi_0 \|}_{L^{p'} \cap H^s} 
  + \int_0^t \langle r \rangle^{-d/(2q)} {\| (u_1)_r - (u_2)_r \|}_{H^s} \, dr,
\end{align*}
which implies \eqref{lemHsdiffconc} via Gronwall's inequality since $d/(2q)>1$.
\end{proof}

\begin{proof}[Proof of Lemma \ref{lemL2diff}]
Let $\phi\in L^2_\gamma\cap H^\gamma$. We use the notations of the proof of Lemma \ref{lemHsdiff} and proceed similarly. As in \eqref{prHsdiff1}, we have
\begin{align}
\nonumber
{\| u_1-u_2 \|}_{L^2_\gamma} & \lesssim
  \int_{0}^{t} {\big\| e^{-i(t-r)H} \big(W^{\psi}_{r}(u_1)_r - W^{\varphi}_{r}(u_2)_r\big) \big\|}_{L^2_\gamma} 
  \, dr 
\\
& \lesssim
  \int_{0}^{t} \Big( \langle t-r\rangle^\gamma \big\|W^{\psi}_{r}(u_1)_r - W^{\varphi}_{r}(u_2)_r \big\|_{H^\gamma} \notag \\
  &\qquad \qquad\qquad+\big\|W^{\psi}_{r}(u_1)_r - W^{\varphi}_{r}(u_2)_r \big\|_{L^2_\gamma}\Big)
  \, dr ,
\label{prL2diff1} 
\end{align}
where we used Lemma \ref{lemL2gamma} (with $v=0$) in the second inequality. It follows from \eqref{prHsdiff10}--\eqref{prHsdiff3} that
\begin{align*}
&\big\|W^{\psi}_{r}(u_1)_r - W^{\varphi}_{r}(u_2)_r \big\|_{H^\gamma}\\
&\lesssim \langle r \rangle^{-d/(2q)} \big( {\| \psi_0 - \varphi_0 \|}_{L^{p'} \cap H^\gamma}\|\phi\|_{H^\gamma}+{\| (u_1)_r-(u_2)_r \|}_{H^\gamma}\big).
\end{align*}
Applying Lemma \ref{lemHsdiff}, this gives
\begin{align}
\big\|W^{\psi}_{r}(u_1)_r - W^{\varphi}_{r}(u_2)_r \big\|_{H^\gamma}\lesssim\langle r \rangle^{-d/(2q)} {\| \psi_0 - \varphi_0 \|}_{L^{p'} \cap H^\gamma}\|\phi\|_{H^\gamma}.
\end{align}

For the $L^2_\gamma$-norm, using Lemma \ref{lemL2gamma} and Lemma \ref{lemWtdiff}, we have
\begin{align}
& {\big\| (W^{\psi}_{r} - W^{\varphi}_{r})(u_1)_r \big\|}_{L^2_\gamma} 
  \notag \\
  & \lesssim  \big\| (W^{\psi}_{r} - W^{\varphi}_{r})\big\|_{L^\infty} {\| (u_1)_r \|}_{L^2_\gamma}
  \notag \\
    & \lesssim \varepsilon \langle r \rangle^{-\frac{d}{q}+\gamma} {\| \psi_0 - \varphi_0 \|}_{L^{p'}\cap H^\gamma} \|\phi\|_{H^\gamma}+\varepsilon \langle r \rangle^{-\frac{d}{q}} {\| \psi_0 - \varphi_0 \|}_{L^{p'}\cap H^\gamma}{\| \phi \|}_{L^2_\gamma}. \label{prL2diff2} 
\end{align}

Next, we obtain from Lemma \ref{lem:Wt} that
\begin{align}
 \big\| W^{\varphi}_{r} ((u_1)_r-(u_2)_r) \big\|_{L^2_\gamma} & \lesssim \|W^{\varphi}_{r}\|_{L^\infty}  \big\|(u_1)_r-(u_2)_r \big\|_{L^2_\gamma}\notag \\
& \lesssim \varepsilon^2 \langle r \rangle^{-d/q}  {\| (u_1)_r-(u_2)_r \|}_{L^2_\gamma}. \label{prL2diff3} 
\end{align}

Putting together \eqref{prL2diff1}-\eqref{prL2diff3} we have obtained
\begin{align}
\nonumber
{\| u_1-u_2 \|}_{L^2_\gamma} 
& \lesssim  \langle t \rangle^\gamma
  \int_{0}^{t}  \big(  \langle r \rangle^{-d/(2q)} + \langle r \rangle^{-\frac{d}{q}+\gamma} \big) {\| \psi_0 - \varphi_0 \|}_{L^{p'} \cap H^\gamma} \|\phi\|_{H^\gamma} \, dr  \notag \\
&\quad+  \varepsilon  \int_{0}^{t}   \langle r \rangle^{-d/q}  {\| \psi_0 - \varphi_0 \|}_{L^{p'} \cap H^\gamma} \|\phi\|_{L^2_\gamma} \, dr  \notag \\
  &\quad + \varepsilon^2  \int_{0}^{t}  \langle r \rangle^{-d/q}  {\| (u_1)_r-(u_2)_r \|}_{L^2_\gamma}
  \, dr \notag \\
  & \lesssim {\| \psi_0 - \varphi_0 \|}_{L^{p'} \cap H^\gamma} \big ( \langle t\rangle^\gamma  \|\phi\|_{H^\gamma} + \|\phi\|_{L^2_\gamma} \big)\notag \\
  &\quad + \varepsilon^2  \int_{0}^{t}  \langle r \rangle^{-d/q}  {\| (u_1)_r-(u_2)_r \|}_{L^2_\gamma}
  \, dr ,
\label{prHSdiff4} 
\end{align}
since $d/(2q)$ and $d/q-\gamma>1$. Using again $d/q-\gamma>1$, 
Eq. \eqref{prHSdiff4} implies \eqref{lemL2diffconc} via Gronwall's lemma.
\end{proof}

\medskip

\section{
Proof of Theorem \ref{thm:max-vel-x-Ht}}\label{app:max-vel_Ht}

{In this appendix we give the proof of Theorem \ref{thm:max-vel-x-Ht} 
which is an improved version of the maximal velocity bounds for linear time-dependent potentials in \cite{ArbPusSigSof}; 
see Theorem 3.3 there.
We consider time-dependent hamiltonians of the form 
\begin{align}\label{Ht}
H_t:=H+W_t,
\end{align}
where $H=-\frac12 \Delta+V(x)$, with $V$ real and satisfying \eqref{V-cond} 
and $W_t(x)=W(x, t)$, a real, time-dependent bounded potential satisfying 
\begin{align}\label{Wt-cond}
\int_0^\infty \int_t^\infty \|\p_x^\al W_r\|_{L^\infty} drdt <\infty 
\, \text{ with either } \, 0 \le |\alpha| \le 1 \text{ or } 1\le|\al|\le 2.
\end{align}
The main difference with respect to \cite{ArbPusSigSof} is the quantification of the dependence 
of the bounds on the norm in \eqref{Wt-cond}.
The assumption \eqref{Wt-cond} in fact is weaker than the assumption made in \cite{ArbPusSigSof}
and, while this may still not be optimal, such an improvement over \cite{ArbPusSigSof} is necessary 
to obtain the results in the present paper.

We prove Theorem \ref{thm:max-vel-x-Ht} in the case where
\begin{equation}\label{eq:def_wt0}
w_t :=\int_t^{\infty} \| W_r\|_{W^{1,\infty}} dr
\end{equation}
is integrable and explain next how to modify the proof in the case where
\begin{equation}\label{eq:def_wt'}
w_t ':= \max_{1\le|\alpha|\le2} \int_t^{\infty} \|\p_x^\al W_r\|_\infty dr
\end{equation}
is integrable.

\subsection{Preliminary estimates}

We will use the following notation: 
\begin{equation}\label{eq:nota}
A_\rho^\pm:=\{x\in \R^d: \pm |x|\ge \pm\rho\},\quad \chi^-_b:=\chi_{A_{b}^-},
\end{equation}
and
\begin{equation*}
 x_{ts} :=s^{-1}(\x -a -v t),
\end{equation*}
and the convention that 
$A\,\dot\le\, B$ and $A\,\dot\ls\, B$ mean that for any integer $n>0$, there is $C_n>0$ s.t. 
$A \le B+ C_n s^{-n}$ and $A \ls B+ C_n s^{-n}$, respectively. 
Recall that $k_I$ has been defined in \eqref{kg}. Given $I$ a bounded open interval, 
we fix $c> v > k_I$. 
and let $\cF\subset C^\infty(\R;\R)$ be the set of functions $f\ge0$, 
supported in $\R^+$ and satisfying $f(\lam)=1$ for $\lam\ge c-v$, 
and $f^\prime\ge 0$, with $\sqrt{f'}\in C^\infty$.
We say that $u$ is admissible if $u$ is a smooth function such that 
$\mathrm{supp}(u)\subset (0,c-v)$ and $\sqrt{u}\in C^\infty$.

In what follows we use the notation $p:=-i\nabla$.

\begin{lemma}\label{lem:p-est}   
Let $I$ be a bounded open interval, $g\in C_0^\infty(I;\R)$,  $f \in \cF$ and $u^2=f'$.
Then there is $\tilde u$, with $\tilde u^2$ admissible, s.t.
\begin{align}\label{p-est}
\|p u(x_{ts})g(H)\psi\|&\,  \dot\le \, k_I\| u(x_{ts})g(H)\psi\|+ s^{-1}\| \tilde u(x_{ts})g(H)\psi\|.
\end{align}
\end{lemma}

\begin{proof}  
We write $g(H)=\tilde g(H)g(H)$, with $\tilde g\in C^\infty_0(I;\R)$ and $\tilde g=1$ on $\mathrm{supp}\, g$. 
Commuting $\tilde g(H)$ to the left, we find
\begin{align}\label{p-rel} 
p u(x_{ts})g(H)&=  p\tilde g(H) u(x_{ts})g(H)+p [u(x_{ts}), \tilde g(H)] g(H). 
\end{align}
Now, it follows from \eqref{comm-exp-x} in Appendix \ref{sec:commut} that  
\begin{equation*}
p[u(x_{ts}), \tilde g(H)]=\sum_{k=1}^{n-1}\frac{s^{- k}}{k!} p B_k u^{(k)}(x_{ts}) +O(s^{-n}),
\end{equation*}
for any $n$, with $p B_k$ bounded. 
Taking the norm of \eqref{p-rel} applied to $\psi$ and using $\| p\tilde g(H)\|\ls k_I$, 
we obtain
\begin{align*}
\|p u(x_{ts})g(H)\psi\|\,  \dot\le \, k_I\| u(x_{ts})g(H)\psi\|+ s^{-1} \sum_{k=1}^{n-1}\| u^{(k)}(x_{ts})g(H)\psi\|.
\end{align*}
Since $u^{(k)}$ are smooth and $\mathrm{supp}\, u^{(k)}\subset(0,c-v)$, one easily verifies that 
\eqref{p-est} holds for some admissible $\tilde u$.
\end{proof}

\medskip
\subsection{Proof of Theorem \ref{thm:max-vel-x-Ht}}

Note that Eq. \eqref{gt} implies that 
\begin{align}\label{g+gt0} 
g_+(H)-g_t(H) = -i \int_t^\infty U_r^{-1}[g(H), W_r]U_rdr
\end{align}
and, therefore,
\begin{align}\label{g+gt} 
\big\| g_+(H)-g_t(H) \big \| 
  \lesssim w_t,
\end{align}
which, together with $U_t g_t(H)=U_tU_t^{-1}g(H)U_t=g(H)U_t$,  implies
\begin{align} \label{PTR} 
\big \| U_t g_+(H) - g(H)U_t \big \| \lesssim w_t.
\end{align}

\begin{proof}[Proof of Theorem \ref{thm:max-vel-x-Ht}]
Let 
\begin{equation}\label{defpsit}
\psi_t := U_tg_+(H)\phi_0, \qquad \phi_0 := \chi^-_b \phi,
\end{equation}
and
\begin{align}\label{propag-obs1}
\Phi_{ts} & = f(x_{ts}), 
\end{align}
for some $ f \in \cF$, with $0\le t\le s$. Let $\langle \Phi_{ts} \rangle_t := \langle \psi_t , \Phi_{ts} \psi_t \rangle$. We have
 \begin{align}\label{Heis}
\begin{split}
& \p_t\lan \Phi_{ts}\ran_t  = \lan \psi_t , D\Phi_{ts} \psi_t \ran,
  \quad D\Phi_{ts}:=i[H,\Phi_{ts}]+\frac{\partial}{\partial t}\Phi_{ts}.
\end{split}
\end{align}
We compute $D\Phi_{ts}$. First, we have
\begin{align} \label{dt-Phi}
\frac{\partial}{\partial t}\Phi_{ts}=-s^{-1}v \,f^\prime(x_{ts}).
\end{align}
Then, letting $A : = \frac12(p\cdot (\n\x)+(\n\x)\cdot p)$, factorizing $f^\prime= u^2$ and using that $[H_t, \x]=A$ and
$[[A,u],u]=0$, one verifies that
\begin{align} \label{H-Phi-comm}
i[H_t,\Phi_{ts}]=\frac{i}{2}[p^2,\Phi_{ts}]
&=\frac12 s^{-1}(A f^\prime(x_{ts})+f^\prime(x_{ts})A)
\notag
\\ 
& = s^{-1}\,u(x_{ts})\,A\,u(x_{ts}).
\end{align}
Together with \eqref{dt-Phi}, this yields:
\begin{align} \label{DPhi-expr-x}
D\Phi_{ts}= s^{-1}\,u(x_{ts})\,(A-v)\,u(x_{ts}).\end{align}

For convenience, let
$R := \psi_t - g(H)U_t\phi_0$. Note that, by \eqref{PTR}-\eqref{defpsit} we have
\begin{equation}\label{pR0} 
R = O(w_t)\phi_0.
\end{equation} 
Moreover, we can also see that
\begin{equation}\label{pR}
{\| p R \|} = O(w_t)
\end{equation} 
as follows:
from \eqref{defpsit} we have $pR = p (U_t g_+(H) - g(H)U_t) \phi_0$,
and from \eqref{g+gt0} we get
\begin{align}\label{pR1} 
p (U_t g_+(H) - g(H)U_t) = -i p \, U_t \int_t^\infty U_r^{-1}[g(H), W_r]U_rdr;
\end{align}
then, using the same notation in \eqref{p-rep}
(that is, $S:=(H+c)^{1/2}$ with $c:=\inf H + 1$,
and 
$B':=p(H+c)^{-1/2}$) we can control $p=B'S$ 
by $S$, apply the commutator bound \eqref{p-U-comm-est} and its analogue for the inverse $U_r^{-1}$,
and, recalling the definition of $w_t$ in \eqref{eq:integr_wt}, we arrive at \eqref{pR}.

Using the above notation for $R$, from the formula \eqref{Heis}, we can write
\begin{align}
\label{Ht-pf1}
\begin{split}
& \p_t\lan \Phi_{ts}\ran_t 
  = \lan g(H)U_t\phi_0, D\Phi_{ts} g(H)U_t\phi_0\,\ran 
  \\
  & + \lan R, D\Phi_{ts} g(H) U_t \phi_0\,\ran 
  +  \lan g(H)U_t\phi_0, D\Phi_{ts} R\,\ran
  +  \lan R, D\Phi_{ts} R\,\ran.
\end{split}
\end{align}

Now, we claim that, with $k_I$ defined in \eqref{kg}, there is $C>0$ s.t.  
\begin{align}\label{A-est}\notag 
g(H)u(x_{ts})\,A\,u(x_{ts})g(H)\,\dot\le\, k_I & g(H)u(x_{ts})^2g(H)
\\
& + C  s^{-1}g(H)\tilde u(x_{ts})^2g(H),
\end{align} 
where $\tilde u$ is an admissible function. 
To see this, we first estimate 
\begin{align} 
\label{A-est2}|\lan \psi, g(H)u(x_{ts})&\,A\,u(x_{ts})g(H) \psi\ran|
\notag
\\ 
& \le \|\n\x u(x_{ts})g(H)\psi\|\|p u(x_{ts})g(H)\psi\|.
\end{align} 
This inequality, together with \eqref{p-est}, gives 
\begin{align}\label{p-est1}
|\lan \psi, \, & g(H)u(x_{ts})\,A\, u(x_{ts})g(H) \psi\ran|\notag
\\ 
& \, \dot\le \,  \| u(x_{ts})g(H)\psi\| \Big(k_I \| u(x_{ts})g(H)\psi\|+s^{-1}\| \tilde u(x_{ts})g(H)\psi\|\Big),
\end{align} 
which implies \eqref{A-est}.
Now, using \eqref{A-est}, together with \eqref{DPhi-expr-x} 
and the definitions $u(x_{ts})^2=f'(x_{ts})$ and $h(x_{ts}):=\tilde u(x_{ts})^2$, 
we obtain
\begin{align}\label{DPhi-est1}
g(H)&D\Phi_s(t)g(H)\notag
\\ 
&  \dot\le \, (k_I -v)s^{-1} g(H)f'(x_{ts}) g(H) + C s^{-2} g(H)h(x_{ts}) g(H).
\end{align}

Next, we claim that the three terms in the second line of \eqref{Ht-pf1} are all at least 
$O(s^{-1} w_t)$. 
For the first term this bound follows using \eqref{DPhi-expr-x}, \eqref{pR0} and \eqref{p-est}:
\begin{align*}
& \big| \lan R, D\Phi_{ts} g(H) U_t \phi_0\,\ran \big| 
  \\
  & =
  s^{-1} \big| \lan R, u(x_{ts}) A u(x_{ts}) \, g(H) U_t \phi_0\,\ran 
 - v \lan R, u^2(x_{ts}) g(H) U_t \phi_0\,\ran \big| 
  \\
  & \lesssim s^{-1} \| R \| \big( \| \phi_0 \| + \| p u(x_{ts}) g(H) U_t \phi_0 \| \big) 
  = s^{-1} O(w_t).
\end{align*}
The second term can be treated identically. 
For the last term on the right-hand side of \eqref{Ht-pf1} we can 
use \eqref{pR} to get an even better bound of $O(s^{-1}w_t^{2})$:
\begin{align*}
& \big| \lan R, D\Phi_{ts} R \,\ran \big| 
  \lesssim s^{-1} \| R \| \big( \| R \| + \| p R \| \big) = s^{-1} O(w_t^2).
\end{align*}

Going back to \eqref{Ht-pf1}, using also \eqref{DPhi-est1} to bound the first terms on the r.h.s., 
we get, for some admissible function $\tilde f$,
\begin{align}
\notag
\p_t\lan \Phi_{ts}\ran_t 
& = \lan g(H)U_t\phi_0, \, D\Phi_{ts}g(H)U_t\phi_0\,\ran 
  + O\big(s^{-1} w_t ) 
  \\
\notag
&\dot\le \, (k_I-v)s^{-1} \lan g(H)U_t\phi_0, \,f'(x_{ts})\,g(H)U_t\phi_0\ran 
  \\
\label{DPhi-est1-Ht}
  & + Cs^{-2} \lan g(H)U_t\phi_0, \,\tilde f(x_{ts})\,g(H)U_t\phi_0\ran
  + O(s^{-1} w_t). \end{align}

Next, 
passing back to $\psi_t$ by using the pull-through relations in the opposite direction 
and the fact that $\tilde{f}$ is bounded, we obtain
\begin{align} \label{DPhi-est2-Ht}
\p_t\lan \Phi_{ts}\ran_t \, \le \,& (k_I-v)s^{-1} \,\lan f'(x_{ts})\ran_t + Cs^{-2}   + C s^{-1} w_t.
 \end{align}
Since $v>k_I$, we can drop the first term on the r.h.s. Using 
the definition $\Phi_{ts}:= f(x_{ts})$, the conditions \eqref{eq:integr_wt} and $s\geq t$, we find 
\begin{align} \label{propag-est2-Ht} 
\lan f(x_{ts})\ran_t\le \lan f(x_{0s})\ran_0
+ C s^{-1}. 
\end{align}
For the first term on the r.h.s., we claim that, for $0 < \beta < 1$,
\begin{align} \label{local-est3-Ht} 
\lan f(x_{0s})\ran_0=O(s^{2\beta-2}) + O( w_{s^\beta}^2 ).
\end{align} 
To prove this estimate,  
we recall  $\psi_t:=U_t g_+(H)\chi^-_b\phi$, note that 
\begin{equation*}
\lan f(x_{0s})\ran_0=\| \chi(x_{0s})\psi_0\|^2,
\end{equation*} 
with $\chi^2= f$, and pass from $g_+(H)$ to $g_{s^\beta}(H):=U_{s^\beta}^{-1} g(H) U_{s^\beta}$, 
with $\beta<1$, paying with the error $O(w_{s^{\beta}})$ (see \eqref{PTR}):
\begin{equation*}
\chi(x_{0s})\psi_0=\chi(x_{0s})g_+(H)\chi^-_b\phi=\chi(x_{0s})g_{s^\beta}(H)\chi^-_b\phi+O(w_{s^{\beta}}).
\end{equation*} 
In Lemma \ref{lem:chi-g+-est} of Appendix \ref{sec:commut} we show that 
\begin{align*}
\chi(x_{0s})g_{s^\beta}(H)\chi^-_b =O(s^{ \beta-1}).
\end{align*} 
This, together with the previous estimate  yields (after squaring up) \eqref{local-est3-Ht}. Finally,
\eqref{local-est3-Ht} and \eqref{propag-est2-Ht} imply
\begin{align*}
\lan f(x_{ts})\ran_t\le Cs^{-1} + C s^{2\beta-2} + C w_{s^\beta}^2
\end{align*} 
which, in view of the definition of $f$, gives, after setting $s=t$,
Theorem \ref{thm:max-vel-x-Ht} in the case of integrable $w_t =\int_t^{\infty} \| W_r\|_{W^{1,\infty}} dr$.

The proof in the case of integrable $w_t' = \max_{1\le|\alpha|\le2} \int_t^{\infty} \|\p_x^\al W_r\|_\infty dr$
is identical, the only difference being that the estimate
\begin{equation*}
\big\| [g(H),W_r] \big\| \lesssim \| W_r  \|_{L^\infty} 
\end{equation*}
used to prove \eqref{g+gt} is replaced by
\begin{equation*}
\big \| [g(H), W_r] \big \| \lesssim \max_{1\le|\alpha|\le2} \|\p_x^\al W_r \|_{L^\infty} ,
\end{equation*}
see Lemma \ref{lem:g-W-comm} below.
\end{proof}

\medskip
\section{Commutator expansions and localization estimate} \label{sec:commut}

First, we state commutator expansions and estimates, first derived in \cite{SigSof} 
and then improved in \cite{Skib, HunSig1,HunSigSof} (see also \cite[Appendix B]{ArbPusSigSof}).
We follow \cite{HunSig1} and refer to this paper as well as to \cite{ArbPusSigSof} for 
details and references.
To begin with, we mention that, by the Helffer-Sj\"ostrand formula, 
a function $f(A)$ of a self-adjoint operator $A$ can be written as
\begin{align} \label{fA-repr}
&f(A)=\int d\widetilde f(z)(z-A)^{-1},
\end{align}
where $\widetilde f(z)$ is an almost analytic extension of $f$ to $\C$ supported in a complex neighborhood of $\supp f$.
For $f\in C^{n+2}(\R)$, we can choose $\widetilde f$ satisfying the estimates (see (B.8) of \cite{HunSig1}):
\begin{align} \label{tildef-est}
&\int |d\widetilde f(z)||\im(z)|^{-p-1}\ls \sum_{k=0}^{n+2}\|f^{(k)}\|_{k-p-1},\end{align}
where $\|f \|_{m}:=\int \x^m |f(x)|.$ 

The essential commutator estimates are incorporated in the following lemma. We refer the reader
to e.g. \cite[Lemma A.1]{ArbPusSigSof} for a proof. 

\begin{lemma}\label{lem:commut-exp} 
Let $f\in C^\infty(\R)$ be bounded, 
  with $\sum_{k=0}^{n+2}\|f^{(k)}\|_{k-2}<\infty$, for some $n\ge 1$.
Let  $x_s=s^{-1}(\x-a)$ for $a >0 $ and
$1\le s<\infty$. Suppose that $H$ satisfies \eqref{v-cond} and let 
 $g\in C_0^\infty(\R)$. Then, for any $n\ge 1$, 
 \begin{align}
 \label{comm-exp-x} [g(H), f(x_s)]&=  \sum_{k=1}^{n-1}\frac{s^{- k}}{k!}B_k f^{(k)}(x_s) +O(s^{-n}), 
 \end{align}
 uniformly in $a\in \R$, where $H^j B_k, j=0, 1, k =1, \dots, n-1,$ 
 are bounded operators and $\|H^j O(s^{-n})\|\ls s^{-n}, j=0, 1$. 
 For $n=1$, the sum on the r.h.s. is omitted. 
\end{lemma}

\smallskip
Now, we prove a localization 
  estimate used at the of the proof of Theorem \ref{thm:max-vel-x-Ht}. 
Recall the definition $g_{s^\beta}(H):=U_{s^\beta}^{-1}g(H) U_{s^\beta}$.

\begin{lemma}\label{lem:chi-g+-est} 
Suppose that
\begin{equation}\label{eq:def_wt'1}
\int_0^\infty w_t^{(1)}dt<+\infty , \quad w_t^{(1)} := \max_{|\alpha|=1} \int_t^{\infty} \|\p_x^\al W_r\|_\infty dr.
\end{equation}
With the notations of Eq. \eqref{eq:nota}, we have 
\begin{align}\label{eq:commut_fin}
\|\chi(x_{0s})g_{s^\beta}(H)\chi^-_b\| =O(s^{\beta-1}).
\end{align}
\end{lemma} 

\begin{proof}
Let $\chi\equiv \chi(x_{0s})$. 
Using $\chi \chi^-_b=0$, we write 
\begin{align}\label{chi-g+-exp}
\chi g_{s^\beta}(H)\chi^-_b
 = [\chi, U_{s^\beta}^{-1}]g(H) U_{s^\beta}\chi^-_b+&U_{s^\beta}^{-1} [\chi, g(H)] U_{s^\beta}\chi^-_b\notag
 \\
& + U_{s^\beta}^{-1}g(H)[\chi,  U_{s^\beta}]\chi^-_b.
\end{align}
Since $[\chi,  U_{s^\beta}]= U_{s^\beta}(U_{s^\beta}^{-1} \chi U_{s^\beta}-\chi)
= U_{s^\beta}\int_0^{s^\beta}\p_r(U_r^{-1} \chi U_r)dr$ and $\p_r(U_r^{-1} \chi U_r)=i U_r^{-1} [H_r, \chi] U_r$ ,
we have
\begin{align}\label{chi-U-comm-exp'}
[\chi, U_{s^\beta}]=i U_{s^\beta}\int_0^{s^\beta} U_r^{-1} [H_r, \chi] U_rdr.
\end{align}
Note that $[H_r, \chi]=-i p\n\chi + \frac12(\Delta\chi)$. 
We control $p$ by $S^{-1}=(H+c)^{-1/2}$, 
with $S=(H+c)^{1/2}$ and $c:=\inf H + 1$:
 \begin{align}\label{p-rep}p=S B=B' S,\end{align}  
 where $B:=(H+c)^{-1/2}p$ and $B':=p(H+c)^{-1/2}$, bounded operators.  
 Eq. \eqref{chi-U-comm-exp'}, together with the last two relations, gives
\begin{align}\label{chi-U-comm-exp''}
 [\chi, U_{s^\beta}]&= U_{s^\beta} \int_0^{s^\beta} U_r^{-1} \big(S B\n\chi
   + i\frac12(\Delta\chi) \big) U_r dr.
\end{align}
Next, we commute $(H+c)^{1/2}$ to the left. To this end, we apply the equation
\begin{align}\label{p-U-comm-est}
[S, U_{r}]=O(1),
\end{align}
which we now prove. 
First, we write $S=(H+c)^{1/2}=(H+c)(H+c)^{-1/2}$ and use the explicit formula 
$(H +c)^{-s}:=c'  \int_0^\infty (H +c+\om)^{-1} d\om/\om^s, $ 
 where $s\in (0, 1)$ and $c' :=[\int_0^\infty (1+\om)^{-1} d\om/\om^s]^{-1}$, to obtain $ [W_{r}, (H+c)^{1/2}]=O(w_r^{(1)})$. This implies the estimate $ [H_{r}, (H+c)^{1/2}]= [W_{r}, (H+c)^{1/2}]=O(w_r^{(1)})$, which, together with the fact that $w_r^{(1)}$ is integrable  
and the relation 
\begin{align}\label{p-U-comm-est'}
[S, U_{r}]=U_{r}\int_0^{r}i U_{r'}^{-1} [H_{r'}, S] U_{r'}dr',
\end{align} 
  yields \eqref{p-U-comm-est}.

Commuting $S$ in Eq. \eqref{chi-U-comm-exp''} to the left (either twice through $U_r^{-1}$ and $U_{s^\beta}$, or once through $U_{s^\beta}U_r^{-1}=U(s^\beta, r)$) and using  \eqref{p-U-comm-est}, $\n\chi=O(s^{-1})$ and $\Delta\chi=O(s^{-2})$, gives
\begin{align}
\notag [\chi, U_{s^\beta}] \notag &=S O(s^{\beta-1})+  \int_0^{s^\beta}  \big( O(s^{-1})  
   + O(s^{-2}) \big) U_r dr\\   
\label{chi-U-comm-exp} 
& =S O(s^{\beta-1}) + O(s^{\beta-1}).
\end{align}

A similar estimate holds for $[\chi, U_{s^\beta}^{-1}]= - [\chi, U_{s^\beta}]^*$:
\begin{align}
\label{chi-Uinv-comm-est}
[\chi, U_{s^\beta}^{-1}]=O(s^{\beta-1})S+O(s^{\beta-1}).
\end{align}
 
Now, the second term on the r.h.s. of the above relation produces the right bound, $O(s^{-1+\beta})$ 
and so does the first term multiplied by $g(H)$, as $(H+1)^{1/2} g(H)$ is a bounded operator. 
This shows that the first term on the r.h.s. of \eqref{chi-g+-exp} is of the order $O(s^{-1+\beta})$. 
The same estimates apply to  the third term on the r.h.s. of \eqref{chi-g+-exp} giving $O(s^{-1+\beta})$. 
For the second term on the r.h.s. of \eqref{chi-g+-exp}, we use  \eqref{comm-exp-x} to obtain $[\chi, g(H)] =O(s^{-1})$.
This proves the lemma.  
\end{proof}

 Finally we prove a lemma used in the proof of Theorem \ref{thm:max-vel-x-Ht} in the case of integrable $w_t' = \max_{1\le|\alpha|\le2} \int_t^{\infty} \|\p_x^\al W_r\|_\infty dr$.

\begin{lemma}\label{lem:g-W-comm} 
We have
\begin{align}\label{g-W-comm} 
\big \| [g(H), W_r] \big \| \lesssim \max_{1\le|\alpha|\le2} \|\p_x^\al W_r \|_{L^\infty} .
\end{align}
\end{lemma}

\begin{proof} 
Using \eqref{fA-repr},  we have
 \begin{align}
\label{comm-exp1'}[W_r, g(H)]&=\int d\widetilde g(z)(z-H)^{-1}[W_r, H](z-H)^{-1}. 
     \end{align} 
Using this, estimate \eqref{tildef-est} for $\widetilde g(z)$ and the fact that $[W_r, H]=\n W_r\cdot\n+\frac12
\Delta W_r$ is $(H+c)^{1/2}$-bounded, where $c:=\inf H+1$ and the estimate 
\begin{equation*}
\|(H+c)^{1/2}R(z)\|\ls |\re z|^{1/2}/|\im z|,
\end{equation*}
we arrive at \eqref{g-W-comm}. 
\end{proof}


\medskip
\begin{thebibliography}{1}




\bibitem{ArbPusSigSof}  J. Arbunich, F. Pusateri, I.M. Sigal and A. Soffer,
Maximal velocity of propagation, Lett. Math. Phys. 111, no. 3, (2021), 1--16. 



\bibitem{Bec}
M. Beceanu,
\newblock \textit{Structure for wave operators om $\R^3$},
\newblock Transactions of the American Journal of Mathematics 136, no. 2, (2014), 255--308.

\bibitem{BeSch}
M. Beceanu and W. Schlag,
\newblock \textit{Structure formulas for wave operators under a small scaling invariant condition},
\newblock J. Spectr. Theory 9, no. 3, (2019), 967--990.


\bibitem{BenPorSchl} 
N.~Benedikter,  M. Porta, B. Schlein, \textit{Effective Evolution Equations from Quantum Dynamics}, SpringerBriefs in Mathematical Physics, 7, Springer, (2016).


\bibitem{BonyFaupSig}  J.-F. Bony,  J. Faupin,  I.M. Sigal,
\textit{Maximal velocity of photons in non-relativistic QED}, Adv. Math. 231, (2012), 3054--3078.

\bibitem{Caz} T. Cazenave, \textit{Semilinear Schr\"odinger Equations}, 
Courant Lecture Notes, 10, AMS, CIMS, 2003.


	\bibitem{exp1}
	M. Cheneau, P. Barmettler, D. Poletti, M. Endres, P. Schau\ss,  	T. Fukuhara,C. Gross, I. Bloch,
	 C. Kollath, and S. Kuhr,  \textit{Light-cone-like spreading of correlations in a quantum many-body system}, Nature 481, (2012), 484--487. 
	


\bibitem{CFKS} H. Cycon, R. Froese, W. Kirsch and B. Simon, \textit{Schr\"odinger Operators}. Texts and Monographs in Physics, Springer Verlag, 1987.


\bibitem{dAN} 
P. d'Ancona and F. Nicola, \textit{Sharp $L^p$ estimates for Schr\"odinger groups}, arXiv:1409.6853.


\bibitem{Dav} E.B. Davies, \textit{Spectral Theory and Differential Operators}. Cambridge University Press, 1995.


\bibitem{DerGer} J. Derezi\'nski and C. G\'erard, {\it Scattering Theory of Classical and Quantum $N$-Particle Systems}. Springer-Verlag, Berlin, 1997.


\bibitem{Dietze}
C. Dietze, \textit{Dispersive estimates for nonlinear Schr\"odinger equations with external potentials},
J. Math. Phys. 62, (2021), 111502. 



\bibitem{ElMaNayakYao}
D. V. Else, F. Machado, Ch. Nayak, and N. Y. Yao, 
\textit{Improved Lieb-Robinson bound for many-body Hamiltonians with power-law interactions},
Phys. Rev. A 101, (2020), 022333. 



\bibitem{FLS1}    
J. Faupin,  M. Lemm, I.M. Sigal 
\textit{Maximal speed for macroscopic particle transport in the Bose-Hubbard model}.  Phys. Rev. Letters 128(15) (2022), arXiv:2110.04313.

\bibitem{FLS2}    
J. Faupin,  M. Lemm, I.M. Sigal, 
\textit{On Lieb-Robinson for the Bose-Hubbard model},  
 Commun. Math. Phys.  394 (3) (2022) 1011-1037, arXiv:2109.04103.
	
\bibitem{Fossetal} 
M.~Foss-Feig, Z.-X.~Gong, C.W.~Clark, and A.V.~Gorshkov, 
\textit{Nearly-linear light cones in long-range interacting quantum systems} Phys.\ Rev.\ Lett. 114 (2015), 157201.
	
	
\bibitem{FincoYajima}
D. Finco and K. Yajima.
\textit{The $L^p$ boundedness of wave operators for Schr\"odinger operators with threshold singularities. 
II. even-dimensional case.} 
Journal of Mathematical Sciences, The University of Tokyo 13, no. 3 (2006), 277--346.


\bibitem{GebNachReSims}
M. Gebert, B. Nachtergaele, J. Reschke, R. Sims, 
\textit{Lieb-Robinson bounds and strongly continuous dynamics for a class of many-body fermion systems in $\R^d$}, 
Ann. Henri Poincar\'e 21, no. 11, (2020), 3609--3637.
 
\bibitem{GHW}
P. Germain, Z. Hani and S. Walsh,
\newblock \textit{Nonlinear resonances with a potential: multilinear estimates and an application to NLS},
\newblock Int. Math. Res. Not. 18, (2015), 8484--8544.

\bibitem{GEi}  C. Gogolin and J. Eisert, \textit{Equilibration, thermalization, and the emergence of statistical mechanics in closed quantum systems}, Rep. Prog. Phys. 79,, (2016), 056001.

\bibitem{HelffSj}  B. Helffer and J. Sj\"ostrand, \textit{Equation de Schr\"odinger avec champ magn\'etique et \'equation de Harper}. In \textit{Schr\"odinger operators}. H. Holden, A. Jensen eds., Lecture Notes in Physics, Vol. 345, Springer Verlag, 1989.


 
 
\bibitem{HeSk} I. Herbst and E. Skibsted, \textit{Free channel Fourier transform in the long-range N-body problem}, J. d'Analyse Math. 65, (1995), 297--332.

\bibitem{HunSig1}  W. Hunziker and I.M. Sigal, \textit{Time-dependent scattering theory of n-body quantum systems}, Rev. Math. Phys. 12, no. 8, (2000), 1033--1084. 

\bibitem{HunSigSof} W. Hunziker, I.M. Sigal and A. Soffer, \textit{Minimal escape velocities}, Comm. Partial Differential Equations 24, (1999), 2279--2295.

\bibitem{IvrSig} V. Ivrii and I.M. Sigal, \textit{Asymptotics of the ground state energies of large Coulomb systems}, Annals of Math. 138, (1993), 243--335.

\bibitem{JenNak} A. Jensen and S. Nakamura, \textit{Mapping Properties of functions of Schr\"odinger operators between $L^{p}$-spaces and Besov Spaces}. Advanced Studies in Pure Mathematics, 23, Spectral and Scattering Theory and Applications, pp. 187-209, 1994.





		\bibitem{KS} T.~Kuwahara and K.~Saito, \textit{Lieb-Robinson bound and almost-linear light-cone in interacting boson systems}, Phys.\ Rev.\ Lett.\ 127, (2021),  070403




\bibitem{LR}
	E.H.~Lieb, and D.W.~Robinson, \textit{The finite group velocity of quantum spin systems}, In Statistical mechanics, 425-431, Springer, Berlin, 1972
	


\bibitem{MatKoNaka}T. Matsuta, T. Koma, S. Nakamura, \textit{Improving the Lieb-Robinson bound for long-range interactions},
 Annales Henri Poincar\'e 18, (2017), 519--528.




	
\bibitem{NRSS}
B.~Nachtergaele, H.~Raz, B.~Schlein, and R.~Sims, 
\textit{Lieb-Robinson Bounds for harmonic and anharmonic lattice systems}, 
Comm.\ Math.\ Phys. 286, no. 3, (2009), 1073--1098.
	
		

\bibitem{NachSim}
B. Nachtergaele and R. Sims, 
\textit{Much ado about something:
Why Lieb-Robinson bounds are useful}, arXiv:1102.0835.

\bibitem{Nak}  S. Nakamura,  \textit{$L^p$-Estimates for Schr\"odinger operators}, 
Proc. Indian Acad. Sci. (Math. Sci.) Vol. 104, No. 4, (1994), pp. 653--666. 






\bibitem{PDdft}
F. Pusateri and I.M. Sigal.
\newblock \textit{Long-Time Behaviour of Time-Dependent Density Functional Theory}
\newblock Arch. Ration. Mech. Anal. 241 (2021), no. 1, 447-473.


\bibitem{RS4}  M.~Reed and B.~Simon.
\newblock {\it Methods of modern mathematical physics. {I}--{IV}}.
\newblock Academic Press, New
  York-London, 1975--1980.

\bibitem{RodSchl}  I. Rodnianski, W. Schlag,
\textit{Time decay for solutions of Schr\"dinger equations with rough and time-dependent potentials}, 
Invent. Math. 155 (3), (2003), 451--513.

\bibitem{SchSurvey}
W. Schlag, \textit{On pointwise decay of waves}. J. Math. Phys. 62, (2021), 061509. 
 
 	
\bibitem{SHOE}
 N.~Schuch, S.K.~Harrison, T.J.~Osborne, and J.~Eisert,
\textit{Information propagation for interacting-particle systems}. Phys.\ Rev.\ A 84, (2011), 032309.
 

\bibitem{Sig} I.M. Sigal, \textit{On long range scattering}. Duke Math. J. 60 (1990) 473--496.





\bibitem{SigSof} I.M. Sigal and A. Soffer, \textit{Local decay and propagation estimates for time-dependent and time-independent Hamiltonians}, 
Preprint, Princeton Univ. (1988)  http://www.math.toronto.edu/sigal/publications/SigSofVelBnd.pdf.

\bibitem{SigSof2} I.M. Sigal and A. Soffer, \textit{Long-range many-body scattering}, Invent. Math. 99, (1990), 115--143.

\bibitem{Skib} E. Skibsted, \textit{Propagation estimates for N-body Schr\"odinger operators}, Comm. Math. Phys. 142, (1992), 67--98.



\bibitem{WH}
		Z.~Wang and K.R.~Hazzard, \textit{Tightening the Lieb-Robinson Bound in Locally Interacting Systems}, PRX
Quantum \textbf{1}, (2020), 010303.
	

\bibitem{Yajima}
K. Yajima.
\newblock \textit{The $W^{k,p}$-continuity of wave operators for Schr\"odinger operators}, 
\newblock J. Math. Soc. Japan 47, no. 3, (1995), 551--581. 



	
		\bibitem{YL}
		C.~Yin and A.~Lucas, \textit{Finite speed of quantum information in models of interacting bosons at finite density}, arXiv:2106.09726.


\end{thebibliography}
\end{document}